\documentclass{amsart}
\usepackage{amsmath}
\usepackage{amssymb}
\usepackage[all]{xy}
\usepackage{comment}
\usepackage{color}

\theoremstyle{plain}
\newtheorem{thm}{Theorem}[section]
\newtheorem{thm*}{Theorem}[section]
\newtheorem{cor}[thm]{Corollary}
\newtheorem{prop}[thm]{Proposition}
\newtheorem{lemma}[thm]{Lemma}

\newtheorem{lemma*}{Lemma}

\theoremstyle{definition}
\newtheorem{defn}[thm]{Definition}
\newtheorem{remark}[thm]{Remark}
\newtheorem{question}[thm] {Question}

\newtheorem*{remark*}{Remark}

\newtheorem{ex}[thm]{Example}

\newtheorem{question*}{Question}

\numberwithin{equation}{thm}

\newcommand{\bM}{\mathbb M}

\newcommand{\bB}{\mathbb B}
\newcommand{\cM}{\mathcal M}

\newcommand{\cN}{\mathcal N}

\newcommand{\bG}{\mathbb G}

\newcommand{\cO}{\mathcal O}

\newcommand{\bZ}{\mathbb Z}
\newcommand{\bF}{\mathbb F}

\newcommand{\cC}{\mathcal C}

\newcommand{\cK}{\mathcal K}

\newcommand{\cI}{\mathcal I}
\newcommand{\cE}{\mathcal E}

\newcommand{\fg}{\mathfrak g}
\newcommand{\fh}{\mathfrak h}

\newcommand{\ol}{\overline}
\newcommand{\ul}{\underline}

\def\sl2{\operatorname{SL_{2(2)}}\nolimits}
\def\Ga2{\operatorname{\mathbb G_{\rm a(2)}}\nolimits}

\def\GL{\operatorname{GL}\nolimits}

\newcommand{\bH}{\mathbb H}
\newcommand{\bN}{\mathbb N}

\newcommand{\bU}{\mathbb U}

\newcommand{\bu}{\bullet}

\date\today

\begin{document}

 \title[Infinite dimensional modules for linear algebraic groups]
 {Infinite dimensional modules for linear algebraic groups} 
 
 \author[ Eric M. Friedlander]
{Eric M. Friedlander$^{*}$} 

\address {Department of Mathematics, University of Southern California,
Los Angeles, CA 90089}
\email{ericmf@usc.edu}

\thanks{$^{*}$ partially supported by the Simons Foundation }



\begin{abstract}
We investigate infinite dimensional modules for a linear algebraic group
$\mathbb G$ over a field of positive characteristic $p$. 
For any subcoalgebra $C \subset \mathcal O(\mathbb G)$ of the coordinate algebra
of $\mathbb G$, we consider the abelian subcategory
$CoMod(C) \subset Mod(\mathbb G)$ and the left exact functor
$(-)_C: Mod(\mathbb G) \to CoMod(C)$ that is right adjoint to the inclusion functor. 
The class of cofinite $\mathbb G$-modules is formulated using finite dimensional
 subcoalgebras of $\mathcal O(\mathbb G)$ and the new invariant of ``cofinite type"
 is introduced.
 
 We are particularly interested in mock injective $\bG$-modules, $\bG$-modules which
 are not seen by earlier support theories.  Various properties of these ghostly
 $\bG$-modules are established.  The stable category $StMock(\bG)$ 
 is introduced, enabling mock injective $\bG$-modules to fit into the framework of  
 tensor triangulated categories.
\end{abstract}

\maketitle


\section{Introduction}

An approach to studying a $\bG$-module $M$ for a connected affine group scheme $\bG$
over a field $k$ of characteristic $p > 0$ is to investigate the restriction of $M$ to Frobenius 
kernels $\bG_{(r)}$ of $\bG$.   From some
points of view, the representation theory of finite group schemes such as $\bG_{(r)}$ resembles
the representation theory of finite groups and thus shares many useful properties.   The technique
of restricting $\bG$-modules to Frobenius kernels has been effective in studying irreducible modules
and standard finite dimensional modules for reductive groups $\bG$ (see, for example, \cite{J}). 
One method for studying the representations of finite group schemes $G$ involves constructing a 
suitable theory of ``supports" for $G$-modules.  This method, beginning with the
consideration of support varieties for elementary abelian 
$p$-groups (see \cite{C}), has been developed for the representation theory
of  finite group schemes (see, for example, \cite{FP2}), linear algebraic groups, various finite dimensional 
algebras (see, for example, \cite{BIKP}, \cite{F-N2}, \cite{NP}), and linear algebraic groups (see \cite{F15}).   

Some aspects of the representation theory of $\bG$ are not seen by support varieties
or by restriction of $\bG$-modules to Frobenius kernels.  Of particular importance are proper mock 
injective modules, modules which are not injective as $\bG$-modules but whose restrictions to 
each Frobenius kernel $\bG_{(r)}$ is an injective $\bG_{(r)}$-module \cite{F18}.
Using the lens of ``stable categories", we showed in \cite{F23} that localizing the category
of bounded cochain complexes $\cK^b(Mod(\bG))$ of $\bG$-modules with 
respect to the bounded derived category of mock injective $\bG$-modules yields a category
(which we denote by $\ol{StMod}(\bG)$ in this paper) which serves as a good
analogue for linear algebraic groups of stable module categories for finite group schemes. 
Indeed, our motivation for this work has been to further our understanding of 
mock injective $\bG$-modules.  In Section \ref{sec:stable}, we formulate some of this
understanding in terms of the thick 
subcategory $StMock(\bG) \hookrightarrow StMod(\bG)$ of bounded complexes associated
to mock injective $\bG$-modules.

In \cite{F18}, we considered linear algebraic groups of exponential type and utilized
the filtration by ``exponential degree" of a $\bG$-module $M$. 
This suggested that a useful approach to studying a $\bG$-module $M$ is to consider
filtrations by $\bG$-submodules whose coaction is constrained to specified subspaces 
of $\cO(\bG)$.   In this paper, we continue this analysis, providing numerous functorial
filtrations of $\bG$-modules for an arbitrary linear algebraic group $\bG$ each of which 
is a filtration by $\bG$-modules (not necessarily the case for the filtration of \cite{F18}).

In Theorem \ref{thm:M-X}, we associate to an arbitrary subspace (i.e., an arbitrary $k$-vector 
subspace) $X \subset \cO(\bG)$ an abelian subcategory 
$i_X: Mod(\bG,X) \hookrightarrow Mod(\bG)$ of the abelian category of $\bG$-modules
 together with a left exact functor $(-)_X: Mod(\bG) \ \to \ Mod(\bG,X)$ that is right adjoint to $i_X$.
 Essentially by construction, $M_X$ is the largest $\bG$-submodule of $M$ whose coaction
 $\Delta_M: M \to M\otimes \cO(\bG)$ factors through $M\otimes X \subset M\otimes \cO(\bG)$.
 We belatedly realized that our construction is an abstraction of J. Jantzen's  ``truncated categories"
 for reductive groups (see \cite[Chap A]{J}).

In Proposition \ref{prop:X-injectivity}, we show for any ascending, converging sequence $\{ X_d \}$
of subspaces of $\cO(\bG)$  that a $\bG$-module $M$  is injective if and only if each $M_{X_d}$ 
is injective in $Mod(\bG,X_d)$.
For $\bG$ of exponential type equipped with the subspaces $\cE(d) \subset \cO(\bG)$ considered in \cite{F18}, 
Proposition \ref{prop:exp-degree} establishes that the filtration $\{ M_{\cE(d)} \}$  of a $\bG$-module
$M$ is the coarsest filtration by $\bG$-submodules subordinate to the filtration 
by exponential degree of $M$ as considered in \cite{F18}.

If a given subspace $X \subset \cO(\bG)$ is the underlying subspace of a subcoalgebra $C$,
then $Mod(\bG,X) \ \subset \ Mod(\bG)$ is naturally identified with $CoMod(C)$, the abelian category
of comodules for $C$.  For a given subspace $X \subset \cO(\bG)$, we consider the smallest 
subcoalgebra  $\cO(\bG)_{\langle X \rangle} \subset \cO(\bG)$
containing $X$; if $X$ is finite dimensional, then so is $\cO(\bG)_{\langle X \rangle}$.   For our
purposes, it is sufficient to consider filtrations of $\bG$-modules provided by an
ascending, converging sequence of 
finite dimensional subcoalgebras of $\cO(\bG)$ such as $\{ \langle X_d \rangle \}$ associated to an 
ascending, converging sequence of finite dimensional subspaces $\{ X_d  \}$
of $\cO(\bG)$.
In Definition \ref{defn:filt-GL}, we provide another construction of ascending, converging 
finite dimensional subcoalgebras $\{ \cO(\bG)_{\leq d,\phi} \}$ of $\cO(\bG)$ 
by first explicitly defining $\{ \cO(GL_N)_{\leq d} \}$ and then using a specified closed embedding 
$\phi: \bG \hookrightarrow GL_N$.  As seen in Proposition \ref{prop:diff-emb}, different closed 
embeddings $\bG$ into general linear groups determine ``comparable" filtrations of $\bG$-modules.   
For a given $\bG$-module $M$, Proposition \ref{prop:adjoints} compares and contrasts 
the filtration $\{ M_{\leq d,\phi} \}$ with the family of restrictions of $\{ M_{|\bG_{(r)}} \}$ to 
Frobenius kernels of $\bG$.   

We say that a $\bG$-module $M$ is ``cofinite" if $M_X$ is finite dimensional for every
finite dimensional subspace $X \subset \cO(\bG)$.  The full subcategory $CoFin(\bG)$
of $Mod(\bG)$ consisting of cofinite $\bG$-modules has various closures properties
(see Proposition \ref{prop:cofinite}) but does not contain all cokernels and is not closed
under tensor products.  
In Definition \ref{defn:poly-growth}, we formulate the ``growth" of cofinite modules; a finer 
invariant, the ``cofinite type" of a $\bG$-module $M$ (denoted 
$\gamma(\bG,\phi)_M)$) is introduced which depends upon a choice of closed embedding
$\phi: \bG \hookrightarrow GL_N$.  Various computations of $\gamma(\bG,\phi)_M$ are given.

In Section \ref{sec:mock}, we investigate mock injective $\bG$-modules.  As observed in
Proposition \ref{prop:property-mock}, the full subcategory
$Mock(\bG)$ of $Mod(\bG)$ whose objects are mock injective $\bG$-modules is an 
exact subcategory with enough injective objects, arbitrary directed colimits, and the ``two out of three"
property.   Moreover, restriction along a closed embedding $\bH \hookrightarrow \bG$ 
determines an exact functor $(-)_\bH: Mock(\bG) \to Mock(\bH)$ and tensor product in
$Mod(\bG)$ determines $\otimes: Mock(\bG) \otimes Mod(\bG) \to Mock(\bG)$.
On the other hand, $Hom_\bG(-,-): Mock(\bG)\times mod(\bG) \to Mod(\bG)$ is typically the 0-pairing
by Theorem \ref{thm:mock-quotients}; in other words, for familiar $\bG$, $Hom_{\bG}(J,M) = 0$
if $J$ is a mock injective $\bG$-module and $M$ is a finite dimensional $\bG$-module.  
Examples \ref{ex:Ga-mock} and \ref{ex:UN-mock} give sample computations 
of the cofinite type of classes of cofinite proper mock injective $\bG$-modules.   Theorem \ref{thm:HNS}
presents a construction due to Hardesty, Nakano, and Sobaje in \cite{HNS} of familes
of mock injective $\bG$-modules

Theorem \ref{thm:categories} presents a ``stable" categorification of this construction,
one which considers triangulated categories in which bounded complexes of injectives 
are set equal to 0.
This formulation involves the tensor triangulated category $StMock(\bG)$, the kernel of the 
quotient functor from $StMod(\bG) \equiv \cK^b(Mod(\bG))/\cI nj(\bG) \to \ol{StMod}(\bG)$.

Before mentioning some open questions in Section \ref{sec:questions}, we provide in 
Theorem \ref{thm:Ga-detect} a necessary and sufficient condition in terms of $\pi$-points 
for a $\bG_a$-module to be injective.

Throughout this paper, the ground field $k$ is assumed to be of characteristic $p > 0$ for 
some prime $p$.  We use $\cO(\bG)$ to denote the coordinate algebra of $\bG$, and $k[\bG]^R$
(respectively, $k[\bG]^L$) the underlying vector space of $\cO(\bG)$ with the right 
(resp., left) regular representation.  For us, an affine group scheme is an affine group scheme 
over $k$ which is of
finite type over $k$; a linear algebraic group is a smooth and connected affine group scheme.  
A closed embedding of affine group schemes will always mean a closed
immersion which is a morphism of group schemes.

We thank Paul Balmer, Bob Guralnick, Julia Pevtsova, Paul Sobaje, and especially 
Cris Negron for helpful insights. 

\vskip .2in


\section{Filtrations by subspaces $X \subset \cO(\bG)$}
\label{sec:X-comodules}

For an affine group scheme $\bG$, we denote by $Mod(\bG)$ the
abelian category of $\bG$-modules; more precisely, the abelian category of 
comodules for $\cO(\bG)$ as a coalgebra over $k$.  Thus, $M \in Mod(\bG)$
is a vector space over $k$ equipped with a right coaction $\Delta_M: M \ \to \  M\otimes \cO(\bG)$
which determines natural (with respect to commutative $k$-algebras
$A$) $A$-linear group actions $\bG(A) \times (A \otimes M) \to (A\otimes M)$.   Unless 
specified otherwise, tensor products are implicitly assumed to be tensor products of 
$k$-vector spaces.

\vskip .1in
\begin{defn}
\label{defn:X-comodule}
Let $i_X: X \subset \cO(\bG)$ be a subspace (that is, a $k$-subspace of $\cO(\bG)$ viewed
as a $k$-vector space).
We define \ $Mod(\bG,X)$ \ to be the full subcategory of $Mod(\bG)$ whose objects
are those $\bG$-modules $M$ whose coaction $\Delta_M$ factors as
$(id_M \otimes i_X) \circ \Delta_{M,X}: M \to M\otimes X \to M \otimes \cO(\bG)$.

We refer to such $\bG$-modules as ``$X$-comodules".
\end{defn}

\vskip .1in

We utilize the following lemma investigating the closure properties of
$Mod(\bG,X) \ \subset Mod(\bG)$.

\begin{lemma}
\label{lem:ker-coker}
Let $X \subset \cO(\bG)$ be a subspace and let $M$ be an $X$-comodule.
\begin{enumerate}
\item
If $j: N \ \to \ M$ is an injective map of $\bG$-modules, then $N$ is also an $X$-comodule.
\item
If $q: M \to Q$ is a surjective map of $\bG$-modules, then  $Q$ is also an $X$-comodule.
\item
If $f: M \to N$ is a map of $\bG$-modules with $N$ an $X$-comodule, then 
the kernel and cokernel of $f$ are $X$-comodules.
\end{enumerate}
\end{lemma}

\begin{proof}
To prove (1), choose a basis $\{ m_\beta, \beta \in I \}$ of $M$ such 
that a subset of this basis is a basis for $N$, and choose a basis $\{ f_\alpha, \alpha \in A \}$ 
of $\cO(\bG)$ such that a subset of this basis is a basis for $X$.
If $m \in M$ is an element of $N$, then $\Delta(m) = \sum_{\beta,\alpha} a_{\beta,\alpha} m_\beta \otimes f_\alpha$
lies both in $N\otimes k[\bG]$ so that $a_{\beta,\alpha} = 0$  unless $m_\beta \in N$ and lies
in $M \otimes X$ so that $a_{\beta,\alpha} = 0$   unless $f_\alpha \in X$. Thus $\Delta(m) \in N\otimes X$
if $m \in N$.

To prove assertion (2), let $j: K \to M$ be the kernel of the surjective map $q: M \to Q$.
Using  assertion (1), we have
the commutative diagram 
\begin{equation}
\label{eqn:q}
\xymatrix{
K \ar[d]_j \ar[r]^-{\Delta_{K,X}} & K \otimes X\ar[r]^-{id_K \otimes i_X}   \ar[d]^{j\otimes id} & 
K \otimes \cO(\bG) \ar[d]^{j\otimes id_{\cO(\bG]}} \\
M \ar[d]_q \ar[r]^-{\Delta_{M,X}} & M \otimes X \ar[r]^-{id_M \otimes i_X}   \ar[d]^{q\otimes id_X} &
 M \otimes \cO(\bG) \ar[d]^{q\otimes id_{\cO(\bG)}} \\
Q  \ar@/_/[rr]_{\Delta_Q} & Q \otimes X   \ar[r]^-{id_Q \otimes i_X} & Q \otimes \cO(\bG).
} 
\end{equation} 
A simple diagram chase for (\ref{eqn:q}) implies that  
$\Delta_Q$ factors uniquely through $Q\otimes X$.

To prove (3), observe that the kernel  $ker\{ f \} \subset M $ is an $X$-comodule by (1) 
and that the quotient $N \twoheadrightarrow coker\{ f \}$ is an $X$-comodule by (2).
\end{proof}

\vskip .1in

We recall that the sum $M_1+M_2\  \subset \ M$ of $\bG$-submodules $M_1, M_2$ of $M$
is the image of $M_1 \oplus M_2 \to M$.

\begin{prop}
\label{prop:colimit}
Let $M$ be a $\bG$-module and $X \subset \cO(\bG)$ be a subspace.  
If $M_1 \subset M, \ M_2 \subset M$ are 
$\bG$-submodules which are $X$-comodules, then $M_1 + M_2 \subset M$
is also an $X$-comodule.
Thus, the category $\chi(M)$ whose objects are $X$-comodules of $M$
and whose maps are inclusions of $\bG$-submodules of $M$ is a filtering subcategory
of $Mod(\bG)$.

Consequently,
\begin{equation}
\label{eqn:bracket-X}
M_X\ \equiv \ \varinjlim_{N \in \chi(M)} N \quad = \ \bigcup_{N \in \chi(M)} N \ \subset \ M
\end{equation}
is well defined as a $\bG$-submodule.  Moreover, \ $M_X  \subset M$ \ is an $X$-comodule,
the largest $X$-comodule contained in $M$.
\end{prop}

\begin{proof}
Recall that $M_1 + M_2 \subset M$ fits in a short exact sequence of $\bG$-modules
\ $0 \to \ M_1\cap \ M_2 \ \to \ M_1\oplus M_2 \ \to \ M_1+M_2 \to 0$.
Since $M_1 \oplus M_2$ is clearly an $X$-comodule whenever $M_1, M_2$ are $X$-comodules,
the first assertion follows from Lemma \ref{lem:ker-coker}.

This implies that the category $\chi(M)$ is filtering;  given two objects $N_1 \subset M$ and $N_2 \subset M$
of $\chi(M)$, both map to $N_1+N_2 \subset M$ which is an object
of $\chi(M)$.   Thus, $\varinjlim_{N \in \chi(M)} N \ \to M$  
equals  the union \ $\bigcup_{N \in \chi(M)} N  \ \subset \ M$.    Recall that $(-) \otimes V$ for a 
given vector space $V$ commutes with filtered colimits of $k$-vector spaces.  Consequently, 
$$\varinjlim_{N \in \chi(M)} \Delta_N: \varinjlim_{N \in \chi(M)} N \ \to \ 
\varinjlim_{N \in \chi(M)} (N \otimes \cO(\bG]) \ = \ (\varinjlim_{N \in \chi(M)} N) \otimes \cO(\bG)$$ 
factors through 
$ \varinjlim_{N \in \chi(M)} N \to (\varinjlim_{N \in \chi(M)} N) \otimes X$.  In other words,
 $M_X$ is an $X$-comodule, the largest $X$-comodule contained in $M$.
\end{proof}

\vskip .1in

Theorem \ref{thm:M-X} introduces the functor $(-)_X: Mod(\bG) \quad \to \quad Mod(\bG,X)$
right adjoint to the natural embedding.

\begin{thm}
\label{thm:M-X}
Let $\bG$ be an affine group scheme of finite type over $k$ and let $i_X: X \subset \cO(\bG)$ be a subspace. 
Denote by  
$$i_{X*}: Mod(\bG,X) \quad \hookrightarrow \quad Mod(\bG)$$
the full subcategory of $Mod(\bG)$ whose objects are $X$-comodules.
\begin{enumerate}
\item
$Mod(\bG,X)$ is an abelian subcategory which is closed under filtering colimits.
\item
Sending a $\bG$-module $M$
to the $\bG$-submodule $M_X$ of $M$ as in (\ref{eqn:bracket-X}) determines a functor
$$(-)_X: Mod(\bG) \quad \to \quad Mod(\bG,X).$$ 
\item
$(-)_X$ is left exact and is right adjoint to the embedding functor $i_{X*}: Mod(\bG,X) \to Mod(\bG)$.
\end{enumerate}
\end{thm}

\begin{proof}
The fact that $Mod(\bG,X)$ is an abelian subcategory of $Mod(\bG)$ follows directly from
Lemma \ref{lem:ker-coker}.    To prove that  $Mod(\bG,X)$ is closed under colimits indexed
by a filtering category $I$, observe that the natural map $\varinjlim_i (M_i \otimes X) \ \to \
\varinjlim_i (M_i) \otimes X$ is an isomorphism.  Thus, if each $M_i$ is an $X$-comodule, 
so is $\varinjlim_i (M_i)$.

To prove functoriality of $(-)_X$, observe that if $f: M \to N$ is a map in $Mod(\bG)$ 
then $f(M_X) \ \subset \ N$ is
contained in $N_X$  by Lemma \ref{lem:ker-coker}(2) and the equality 
$N_X   \ = \ \bigcup_{N^\prime \in \chi_M} N^\prime \ \subset N$ of (\ref{eqn:bracket-X}).  
This equality also shows that $(-)_X$ is left exact.

Functoriality together with (1.3.1) determines the natural inclusion
$$Hom_{Mod(\bG)}(i_{X*}(M),N) \quad \hookrightarrow \quad Hom_{Mod(\bG,X)}(M,N_X)$$
inverse to the  inclusion $Hom_{Mod(\bG,X)}(M,N_X) \ \hookrightarrow \ Hom_{Mod(\bG)}(i_{X*}(M),N)$
and thus a bijection.  This is the natural isomorphism of the asserted adjunction.
\end{proof}

\vskip .1in

Perhaps the simplest example of the functor $(-)_X: Mod(\bG) \quad \to \quad Mod(\bG,X)$
is the case in which $X = k\cdot 1$, the span of $1 \in \cO(\bG)$.  In this case,
$(-)_X \ = \ H^0(\bG,-)$.  Observe that $H^0(\bG,-)$ is left exact for any $\bG$, but is not always exact.

\vskip .1in

\begin{remark}
\label{rem:not-closed}
The full abelian subcategory $Mod(\bG,X) \ \hookrightarrow \ Mod(\bG)$ is 
typically not closed under extensions.  For example, if
$Ext_\bG^1(k,k) \not= 0$, then $Mod(\bG,k\cdot 1)$ is not closed under extensions.
\end{remark}

\vskip .1in

\begin{cor}
\label{cor:enough}
Retain the hypotheses and notation of Theorem \ref{thm:M-X}.  
\begin{enumerate}
\item
If $I$ is an injective $\bG$-module, then $I_X$ is an injective object of $Mod(\bG,X)$.
\item
Similarly, if $J$ is an injective object of $Mod(\bG,X)$ and $Y \subset X$ is a subspace,
then $J_Y$ is an injective object of $Mod(\bG,Y)$.
\item
If $M$ is an $X$-comodule with a given embedding  $M \hookrightarrow I$ into an 
injective $\bG$-module, then $M \hookrightarrow I_X$ in $Mod(\bG,X)$ 
is an embedding of $M$ in an injective object of $Mod(\bG,X)$.
\end{enumerate}
In particular, the abelian category $Mod(\bG,X)$  has ``enough injectives".
\end{cor}

\begin{proof}
If $I$ is an injective $\bG$-module, then the fact that $(-)_X$ has a left adjoint (namely, $i_{X*}(-)$)
which is left exact implies that $I_X$ is an injective object of $Mod(\bG,X)$.   To prove (2), observe
that the embedding $i_{X,Y,*}: Mod(\bG,Y) \hookrightarrow Mod(\bG,X)$ is left adjoint to the 
restriction of $(-)_Y$ to $Mod(\bG,X)$ so that the argument using the existence of a left exact
left adjoint applies to prove (2).
Assertion (3) follows from (1) together with 
the left exactness of $(-)_X$ and the fact that $(-)_X$ restricts to the identify on $Mod(\bG,X)$.
\end{proof}

\vskip .1in

We say a sequence of subspaces \ $\{ X_i \}$ of $\cO(\bG)$ indexed by the non-negative
numbers $i \geq 0$ is an {\it ascending, converging sequence of subspaces} of $\cO(\bG)$
if $X_i \subset X_{i+1}$ for all $i \geq 0$ and if \
 $\bigcup_{i \geq 0} X_i \ = \  \cO(\bG)$.

\begin{prop}
\label{prop:exhaust-X}
Let $\bG$ be an affine group scheme of finite type over $k$ and let $\{ X_i \}$
be an ascending, converging sequence of subspaces of $\cO(\bG)$.
Sending a $\bG$-module $M$ to the sequence of $\bG$-submodules
\begin{equation}
\label{eqn:X-seq}
M_{X_0} \ \subset \ M_{X_1} \ \subset M_{X_2} \subset \cdots \ \subset \ \bigcup_{i \geq 0} M_{X_i} \ = \ M
\end{equation}
is a filtration of $M$, functorial in $M$,  with the property that each $M_{X_i}$ is an $X_i$-comodule.
We say that $\{ M_{X_i} \}$ is an ascending, converging sequence of $\bG$-submodues of $M$.

If $M$ is a finite dimensional $\bG$-module, then $M$ is an $X_i$-comodule for all $i$
sufficiently large.
\end{prop}

\begin{proof}
If the $\bG$-module $M$ is finite dimensional, then $\Delta_M: M \to M\otimes \cO(\bG)$ must have
image in some finite dimensional subspace of $M \otimes X$ and
thus must have image contained in $M\otimes X_i$ for $i$ sufficiently large.   If $M$ 
is an arbitrary $\bG$-module, then $M$
is locally finite so that every $m \in M$ lies in some finite dimensional
$\bG$-submodule $M^\prime \subset M$  and thus must be contained in some $M_{X_i}$ as required.
\end{proof}

\vskip .1in

The following corollary is an immediate consequence of the functoriality of the filtration $M \ \mapsto \ 
\{ M_{X_i} \}$ and the fact that  $\bigcup_{i \geq 0} M_{X_i} \ = \ M$  for any $\bG$-module $M$.

\begin{cor}
\label{cor:functorial}
For any ascending, converging sequence $\{ X_i \}$ of subspaces of $\cO(\bG)$
a map \ $\phi: M \to N$ \ of $\bG$-modules is an isomorphism if and only 
if  \ $(\phi)_{X_i}: M_{X_i} \to N_{X_i}$ is an isomorphism of $\bG$-modules for all $i$.
\end{cor}

\vskip .1in

We argue exactly as in the proof of \cite[Prop 4.2]{F18} to conclude the following
detection of rational injectivity of a $\bG$-module.
 
\begin{prop}
\label{prop:X-injectivity}
Consider an affine group scheme $\bG$ of finite type over $k$ and let $\{ X_i\}$
be an ascending, converging sequence  of subspaces of $\cO(\bG)$.
Then a $\bG$-module $L$ is injective  if and only if $L_{X_i}$ is
an injective object of $Mod(\bG,X_i)$ for all $i \geq 0$.
\end{prop}

\begin{proof}
Granted Corollary \ref{cor:enough}, it suffices to show that if a 
 $\bG$-module $L$ has the property that 
$L_{X_i} \subset L$ is an injective object of $Mod(\bG,X_i)$ for all $i \geq 0$,
then $L$ is an injective $\bG$-module.  
Let $M^\prime \hookrightarrow M$ be an inclusion of $G$-modules and consider
a map  $f^\prime: M^\prime \to L$ of $\bG$-modules.  We inductively 
construct an extension $f_d: M_{X_d} \to L_{X_d}$ of $f^\prime_d: M^\prime_{X_d} \to L_{X_d}$,
the restriction of
$f^\prime$ to $M^\prime_{X_d}$.  Choose $f_d: M_{X_d} \to L_{X_d}$ extending 
$f_d^\prime + f_{d-1}: (M^\prime)_{X_d} + M_{X_{d-1}} \to L_{X_d}$ using the 
injectivity of $L_{X_d}$ as an object of $Mod(\bG,X_d)$ (and taking $M_{X_{-1}} = 0$).  We define
$f: M \to L$  to be $\varinjlim_d  f_d: M_{X_d} \to L_{X_d}$, thereby extending $f^\prime$.
\end{proof}

\vskip .1in

We proceed to give in Proposition \ref{prop:exp-degree} a simple fix for the 
``filtration by exponential degree" of a $\bG$-module $M$ for a linear algebraic group $\bG$ 
of exponential type as formulated in \cite{F18}.  Our modification provides the coarsest filtration 
which is a filtration by $\bG$-submodules subordinate to that of \cite{F18}. 

We recall the definition of a linear algebraic group $\bG$ of exponential type, a class of linear 
algebraic groups for which there is a somewhat explicit geometric description of the support
varieties of its representations.  Let $\cN_p(\fg)$ denote the $p$-nilpotent 
cone of the Lie algebra $\fg = Lie(\bG)$.  Thus, $\cN_p(\fg) \ \subset \ \fg$ is the reduced affine variety
whose set of $K$-points is identified with the elements 
$X \in \fg_K$ such that $X^{[p]} = 0$.  We utilize the notation $\cC_r(N_p(\fg))$
to denote the commuting variety of $r$-tuples of pair-wise commuting, $p$-nilpotent elements of $\fg$.
The condition of Definition \ref{defn:exp-type}(5) for a linear algebraic group $\bG$ requires
that this affine variety has the same $K$-points as
the scheme $V(\bG_{(r)})$ introduced in \cite{SFB1} which represents the functor of 1-parameter 
subgroups of the infinitesimal group scheme $\bG_{(r)}$.

\begin{defn}
\label{defn:exp-type}
\cite[Defn1.6]{F15}
Let $\bG$ be a linear algebraic group with Lie algebra $\fg$.   A structure of exponential type
on $\bG$ is a $\bG$-equivariant morphism of $k$-schemes (with respect to adjoint actions)
\begin{equation}
\label{eqn:Exp}
\cE: \cN_p(\fg) \times \bG_a \ \to \bG, \quad (B,s) \mapsto \cE_B(s)
\end{equation}
satisfying the following conditions for all field extensions $K/k$:
\begin{enumerate}
\item
For each $B\in \cN_p(\fg)(K)$, $\cE_B: \bG_{a,K} \to \bG_K$ is a 1-parameter subgroup.
\item
For any pair of  commuting $p$-nilpotent elements $B, B^\prime \in \fg_K$,
the maps $\cE_B, \cE_{B^\prime}: \bG_{a,K} \to \bG_K$ commute.
\item
For any $\alpha \in K$,  and any 
$s \in \bG_a(K)$, \ $\cE_{\alpha \cdot B}(s) = \cE_B(\alpha\cdot s)$.
\item  Every 1-parameter subgroup $\psi: \bG_{a,K} \to \bG_K$ is of the form 
$$ \cE_{\ul B} \ \equiv \ \prod_{s=0}^{r-1} (\cE_{B_s} \circ F^s)$$
for some $r > 0$, some $\ul B \in \cC_r(\cN_p(\fg_K))$.
\item   The natural
map $\cC_r(\cN_p(\fg)) \to V(\bG_{(r)})$ induces a bijection on $K$-points
sending $\ul B$ to the infinitesimal 1-parameter subgroup $\bG_{a(r),K} \to \bG_{(r),K}$
which factors $\cE_{\ul B} \circ i_r: \bG_{a(r),K} \to \bG_{a,K} \to \bG_{K}$.
\end{enumerate}
\end{defn}

\vskip .1in

Various examples of $\bG$ of exponential type are considered in \cite{Sobj13}; these
include simple classical groups, their standard parabolic subgroups, and the
unipotent radicals of these parabolic subgroups.

\begin{prop} (\cite[Defn 4.5]{F15})
\label{prop:exp-degree}
Let $(\bG,\cE)$ be a linear algebraic group of exponential type.  We define $\cE(\bG)_d
\ \hookrightarrow \ \cO(\bG)$ to be the subspace 
\begin{equation}
\label{eqn:E(d)}
\cE(\bG)_d \ \equiv \ \cE^{*-1}(k[(\cN_p(\fg)][t]_{\leq d}) \quad \subset \quad \cO(\bG)
\end{equation}
where $k[\cN_p(\fg)][t]_{\leq d} \subset k[\cN_p(\fg) \times \bG_a]$ is the 
subspace of polynomials in $k[\cN_p(\fg)][t]$ of degree $\leq d$.  
So defined, 
$\{ M_{\cE(\bG)_d}\}$ is the coarsest filtration of $M$ by $\bG$-modules subordinate
to the ``filtration by exponential degree" of \cite{F18}.

Moreover,  \ $M_{\cE(d)} \otimes N_{\cE(e)} \ \subset \ (M\otimes N)_{\cE(d+e)}$
for every pair of $\bG$-modules $M, N$.
\end{prop}

\begin{proof}
The ``filtration by exponential degree" of \cite[Defn 3.10]{F18} associates to the $\bG$-module $M$
and a positive integer $d$ the subspace $M_{[d]} \subset M$ consisting of elements 
$m \in M$ with the property that $\Delta_M(m) \in M\otimes \cE(\bG)_d$.   By Proposition \ref{prop:colimit},
$M_{\cE(\bG)_d}$ is the largest $\bG$-submodule of $M$ such that $M_{\cE(\bG)_d} \subset M_{[d]}$.

The second statement follows easily from the observation 
the coaction map $\Delta_{M\otimes N}: M\otimes N \ \to \ (M\otimes N) \otimes \cO(\bG)$  arises by 
composing $\Delta_M \otimes \Delta_N$ 
with the product map $\cO(\bG) \otimes \cO(\bG) \ \to \ \cO(\bG)$.
\end{proof}

\vskip .1in

In the following proposition, $\cE_B:\bG_{a,K} \to \bG_K$ is the exponential map determined
by a $K$-point $B$ of $\cN_p(\fg)$ (for some field extension $K/k$)
and the exponential structure $\cE: \cN_p(\bG) \times \bG_a \to \bG$.
For any $s \geq 0$, $u_s: k[t] \to k$ is the $k$-linear map sending $t^i$ to 0 if $i\not= p^s$
and sending $t^{p^s}$ to 1.  We denote by $(\cE_B)_*(u_s): \cO(\bG_K) \to K$  the $K$-linear map given
by the composition $u_s \circ (\cE_B)^*: \cO(\bG_K) \to K[t] \to K$.

We utilize the (``$\pi$-point") support variety $M \mapsto \Pi(\bG)_M$ of \cite{F23} extending the construction 
for finite group schemes given in \cite{FP2}.

We justify saying that $\Pi(\bG)_M$ is the 
``inverse image under the projection" of $\Pi(\bG_{(r)})_{M|\bG_{(r)}}$ by recalling that
$\Pi(\bG)$ (respectively, $\Pi(\bG_r)$) for $\bG$ of exponential type can be identified with 
the projectivization of $\varinjlim_s \cC_s(N_p(\fg))$ (resp., $\cC_r(N_p(\fg))$)
and that there is a natural projection $\varinjlim_s \cC_s(N_p(\fg)) \twoheadrightarrow \cC_r(N_p(\fg))$.

\begin{prop}
\label{prop:support-exp}
Let $(\bG,\cE)$ be a linear algebraic group of exponential type and let $M$ be a $\bG$-module 
with the property that the coaction $\Delta_M: M \to M\otimes \cO(\bG)$ factors through
$M\otimes \cE(\bG)_{p^r-1} \hookrightarrow M \otimes \cO(\bG)$; in other words, assume 
that $M = M_{\cE(\bG)_{p^r-1}}$.   Then, for any
$K$-point $B$ of $\cN_p(\fg)$, $(\cE_B)_*(u_s)$ acts trivially on $M_K$ provided that $s \geq r$.

Consequently, if  $M= M_{\cE(\bG)_{p^r-1}}$, then the support variety $\Pi(\bG)_M$ of $M$ is the 
``inverse image under the projection" of $\Pi(\bG_{(r)})_{M_{|\bG_{(r)}}}$ (containing the
center of the ``projection" $\Pi(\bG) \to \Pi(\bG_{(r)})$, where $M_{|\bG_{(r)}}$ denotes the 
restriction of $M$ to the Frobenius kernel $\bG_{(r)} \hookrightarrow \bG$.
\end{prop}

\begin{proof}
The proposition follows immediately from Proposition \ref{prop:exp-degree} and
\cite[Prop 3.17]{F18}.
\end{proof}

\vskip .2in


\section{Finite dimensional subcoalgebras of $\cO(\bG)$}
\label{sec:coalgebras}

We begin with a proposition which indicates some of the advantages of specializing the 
discussion of $(-)_X: Mod(\bG) \mapsto Mod(\bG,X)$ in Section \ref{sec:X-comodules} 
by requiring $X \subset \cO(\bG)$ to be a subcoalgebra.

\begin{prop}
\label{prop:advantages}
Let $\bG$ be an affine group scheme and let $C \subset \cO(\bG)$ be a subcoalgebra; thus,
$C$ is a subspace of $\cO(\bG)$ with the property that the coproduct $\Delta: \cO(\bG) \to 
\cO(\bG) \otimes \cO(\bG)$ restricts to $C \to C\otimes C$.
\begin{enumerate}
\item 
The full abelian subcategory $Mod(\bG,C)$ of $Mod(\bG)$
(as in Definition \ref{defn:X-comodule}) is naturally identified with the category of comodules
for $C$, \\  $ Mod(\bG,C) \ \simeq \ CoMod(C) $.
 \item
 For any $\bG$-module $M$, the $\bG$-submodule $M_C \subset M$ (as in Proposition \ref{prop:colimit})
 has underlying vector space 
 $\{ m \in M: \Delta_M(m) \in M \otimes C \} \ \subset \ M$.
 \item
 If $C$ contains $1 \in \cO(\bG)$, then $(k[\bG]^R)_C  \quad = \quad C$.
 \end{enumerate}
 \end{prop}
 
 \begin{proof}
 The identification $ Mod(\bG,C) \ \simeq \ CoMod(C) $ follows immediately from the 
 definitions of these categories.
 
 Denote the subspace $\{ m \in M: \Delta_M(m) \in M \otimes C \} \ \subset \ M$ by $M_C^\prime \subset M$.
 We choose a basis $\{ f_s \}$ of $\cO(\bG)$ with the property that $\{ f_s \} \cap C$ is a basis for $C$.
 The definition of $M_C$ immediately implies that $M_C \subset M_C^\prime$.  To prove assertion (2),
 consider an arbitrary element $m \in M_C^\prime$ and write $\Delta(m) = \sum m_i \otimes f_i$
 with each $f_i \in \{ f_s \} \cap C$.  As argued in \cite[I.2.13]{J}, it suffices to show that each 
 $m_i$  satisfies $\Delta_M(m_i) \in M\otimes C$ (i.e., is an element of $M_C^\prime$).
 
 Write $\Delta(m_i) = \sum m_{i_j} \otimes f_{i,j} \in M\otimes \cO(\bG)$ with each $f_{i_j}$ in  $\{ f_s \}$
 and $f_{i_j} \not= f_{i_{j^\prime}}$ whenever $j \not= j^\prime$.  Then
 $$\Delta(\Delta(m)) \ = \ \sum_i (\sum m_{i_j}\otimes f_{i_j} \otimes f_i) \ = \ \sum_i m_i \otimes \Delta(f_i)
 \ = \ \sum_i \sum m_i \otimes f_{i_k} \otimes f_{i_k^\prime}.$$
 Here, $\Delta(f_i) = \sum f_{i_k} \otimes f_{i_k^\prime}$.   Comparing terms 
 in $M \otimes \cO(\bG) \otimes \langle f_i \rangle$
 where $\langle f_i \rangle$ is the 1-dimensional span of $f_i$, we conclude that 
 $\sum_i m_{i_j} \otimes f_{i,j} \ = \ \sum_{\ell, k: f_{\ell_k}^\prime = f_i}  a_\ell m_\ell \otimes f_{\ell,k}$
 for a suitable choice of elements $a_\ell \in k$.
 Since each $f_{\ell,k}$ is in $C$ because each $f_i$ is in the coalgebra $C$, 
 we conclude that each $f_{i,j} \in C$ 
 as required to show that $m_i$ is an element of $M_C^\prime$.
 
 To prove assertion (3), we must show $C \subset (k[\bG]^R)_C$.
 Consider a  basis for $\{ f_s \}$ for $\cO(\bG)$ which 
 includes $1 \in \cO(\bG)$.   Let $I$ be the augmentation ideal of $\cO(\bG)$.   
 Consider an arbitrary element $f \in I$.   As in \cite[I.2.4]{J}, 
 $\Delta(f)$ - $(f\otimes 1 + 1\otimes f) \ \in I\otimes I$.   
 Thus, $\Delta(f) = 1\otimes f + f \otimes 1 + 
 \sum c_{s,s^\prime} f_s\otimes f_{s^\prime}$ with $c_{s,s^\prime} = 0$ if either $f_s$ or $f_{s^\prime}$
 equals 1 and with each $f_{s^\prime} \in C$.  Consequently, if $f \in C \cap I$ and $1 \in C$, then
 $$\Delta(f) \ = \ 1\otimes f  - f\otimes 1- \sum c_{s,s^\prime} f_s\otimes f_{s^\prime} \ \in \
  \cO(\bG) \otimes C,$$
 so that $f \in (k[\bG]^R)_C$.
 \end{proof}
 
 \vskip .1in
 
 We utilize the constructions of Definition 2.2 below to construct the subalgebra $\cO(\bG)_{\langle X \rangle}
 \subset \cO(\bG)$ ``generated" by a subspace $X \subset \cO(\bG)$.

\begin{defn}
\label{defn:closures}
Let $\bG$ be an affine group scheme and $X \subset \cO(\bG)$ be a subspace.
Following \cite{J}, we denote by $k\bG \cdot X$ the smallest $\bG$-submodule
of $\cO(\bG)$ containing $X$
equipped with the right regular regular representation (i.e., $k[\bG]^R$);
thus, $\Delta: \cO(\bG) \to \cO(\bG)\otimes \cO(\bG)$ restricts to \
$k\bG \cdot X \ \to \ k\bG \cdot X \otimes \cO(\bG)$.  

Similarly, we consider the right regular representation of $\bG^{op}$ on 
$\cO(\bG) = \cO(\bG^{op})$, so that $(g,f(x)) \in \bG^{op} \times \cO(\bG)$
maps to $f(gx) \in \cO(\bG)$; this is given by the coproduct
$\tau \circ \Delta: \cO(\bG) \to \cO(\bG)\otimes \cO(\bG) \to  \cO(\bG)\otimes \cO(\bG)$
where $\tau$ switches tensor factors.  For any subspace $Y \subset \cO(\bG)$,
we denote by $k\bG^{op} \cdot Y$  the smallest $\bG^{op}$-submodule of
$\cO(\bG)$ containing $Y$.
\end{defn}

\vskip .1in
The following theorem, called the ``Fundamental Theorem of 
Coalgebras" in \cite{Sw} and the ``Finiteness Theorem" in \cite{smontgom} 
(when stated for arbitrary coalgebras) has the following appealing form
when specialized to the Hopf algebra $\cO(\bG)$.  

\begin{prop}
\label{prop:finite}
Let $\bG$ be an affine group scheme and $X \subset \cO(\bG)$ be a subspace.
Then there is a smallest subcoalgebra of $\cO(\bG)$ containing $X$, 
$\cO(\bG)_{\langle X \rangle} \subset \cO(\bG)$, given by 
$$\cO(\bG)_{\langle X \rangle}  \quad \equiv \quad k\bG^{op} \cdot (k\bG \cdot X).$$
If $X$ is finite dimensional, then  $\cO(\bG)_{\langle X \rangle} $ is also finite dimensional.

Consequently, for any ascending, converging sequence $\{ X_i \}$
of finite dimensional subspaces of $\cO(\bG)$,  there is a smallest ascending, 
converging sequence $\{ \cO(\bG)_{\langle X_i \rangle}  \}$ of finite dimensional subcoalgebras 
of $\cO(\bG)$ satisfying the condition that  $X_i \subset  \cO(\bG)_{\langle X_i \rangle}$.
\end{prop}

\begin{proof}
Observe that $(\cO(\bG) \otimes \cO(\bG)_{\langle X \rangle}) 
\cap (\cO(\bG)_{\langle X \rangle} \otimes \cO(\bG))$
equals $\cO(\bG)_{\langle X \rangle} \otimes \cO(\bG)_{\langle X \rangle}$, 
Thus, to show that $\cO(\bG)_{\langle X \rangle} \subset \cO(\bG)$ is a subcoalgebra, it suffices 
to prove that $\Delta: \cO(\bG) \to  \cO(\bG) \otimes  \cO(\bG)$ restricts to 
\begin{equation}
\label{eqn:restrict}
\cO(\bG)_{\langle X \rangle} \to \cO(\bG) \otimes \cO(\bG)_{\langle X \rangle}, \quad 
\cO(\bG)_{\langle X \rangle} \to \cO(\bG)_{\langle X \rangle} \otimes \cO(\bG).
\end{equation}
The second restriction of (\ref{eqn:restrict} follows from the definition of $G^{op}(-)$ acting on  
$k\bG\cdot X$ via the restriction of its right regular action of $\cO(\bG)$.
Since the right regular actions of $\bG$ and $\bG^{op}$ on $\cO(\bG)$ commute,
$\cO(\bG)_{\langle X \rangle}  \ \equiv \ k\bG \cdot (k\bG^{op} \cdot X)$.  This, together with
the definition of $k\bG\cdot(-)$ applied to $k\bG^{op} \cdot X$, implies 
the first restriction of (\ref{eqn:restrict}

The second assertion of the proposition follows immediately from the fact that if $X$ and $Y$ are
finite dimensional vector subspaces of $\cO(\bG)$, then both $k\bG \cdot X$ and $k\bG^{op} \cdot Y$
are finite dimensional.  (See \cite[I.2.13]{J}.)
\end{proof}

\vskip .1in

In the following proposition, we relate full subcategories of $Mod(\bG)$ to 
subcoalgebras $CoMod(\cO(\bG)_{\langle X \rangle})$ of $\cO(\bG)$.  

\begin{cor}
\label{cor:relate}
Consider an affine group scheme $\bG$ and a full subcategory $\cM \subset Mod(\bG)$.
Let $X_\cM \subset \cO(\bG)$ be the smallest subspace $X \subset \cO(\bG)$ 
such that $M \in Mod(\bG,X)$ for all $M \in \cM$.
Then $\cO(\bG)_{\langle X_\cM \rangle} $ is the smallest subcoalgebra $C \subset \cO(\bG)$ 
with the property that $\cM \ \hookrightarrow \  CoMod(C)$. 
\end{cor}

\begin{proof} 
Assume that $C \subset \cO(\bG)$ is a subcoalgebra such that $\cM \ \hookrightarrow \  
Mod(\bG,C) = CoMod(C)$.  The defining property of $X_\cM$ implies that $X_\cM  \subset C$.  
Consequently, the defining property of $X \mapsto \langle X \rangle$ given in Proposition \ref{prop:finite}
tells us that $\cO(\bG)_{\langle X_\cM \rangle } \ \subset \ C$.
\end{proof}

\vskip .1in

The tensor product of $\cO(\bG)$-comodules involves the product
structure $\mu: \cO(\bG) \otimes \cO(\bG) \to \cO(\bG)$ induced by the diagonal $diag: 
\bG \hookrightarrow \bG \times \bG$.  Thus, unless the subcoalgebra $C \subset \cO(\bG)$
is also a subalgebra, this tensor product does not induce a tensor product structure on $CoMod(C)$.  

The following suggests a useful condition on ascending, converging sequences of subcoalgebras
of $\cO(\bG)$.   

\begin{prop}
\label{prop:prod-coalg}
Let $\bG$ be an affine group scheme.  Let $X, \ Y \ \subset \cO(\bG)$ be subspaces
and define $X \cdot Y \subset \cO(\bG)$ to be the subspace spanned by products $x \cdot y$
with $x \in X, \ y \in Y$.  Consider two $\bG$-modules $M$ and $N$. 
\begin{enumerate}
\item 
The $\cO(\bG)$-module $M_X \otimes M_Y$ is contained in $(M\otimes N)_{X\cdot Y}$
\item
For subcoalgebras $C, \ C^\prime, \ C^{\prime\prime}$ such that the 
multiplication map $\mu: \cO(\bG) \otimes \cO(\bG) \to \cO(\bG) $ restricts to
$C \otimes C^\prime \to C^{\prime\prime}$, there is a natural map of $\bG$-modules
\ $M_C \otimes N_{C^\prime} \ \to \ (M\otimes N)_{C^{\prime\prime}}.$
\end{enumerate}
\end{prop}

\begin{proof}
Assertion (1) follows directly from the fact that the coaction of $\cO(\bG)$ on
$M \otimes N$ is given by the composition of the binary operation of multiplication $\mu$
on the algebra $\cO(\bG)$ and the tensor product $\Delta_M \otimes \Delta_N$ of the coactions of 
$\cO(\bG)$ on $M$ and $N$:
\begin{equation}
\label{eqn:coprod-prod}
\mu \circ (\Delta_M \otimes \Delta_N):  M\otimes N
 \ \to \ M \otimes \cO(\bG) \otimes N \otimes \cO(\bG) \ \to \ M\otimes N \otimes \cO(\bG).
 \end{equation}
  Assertion (2) is a special case of assertion (1).
\end{proof}

\vskip .1in

Let $\cO(\bM_{N, N})$ be the bialgebra given as the coordinate algebra of the affine 
variety of $N \times N$ matrices with monoid structure given by matrix multiplication.
For any $d \geq 0$, we consider the subspace $\cO(\bM_{N, N})_{\leq d} \subset \cO(\bM_{N, N})$
consisting of polynomials in $\{x_{i,j}, 1\leq i,j \leq N \}$ of total degree  $\leq d$.   Since 
$\Delta_{\bM_{N,N}}(x_{i,j}) = \sum_s x_{i,s} \otimes x_{s,j}$,  
$\cO(\bM_{N, N})_{\leq d} \subset \cO(\bM_{N, N})$ is a subcoalgebra.

Equip the coordinate algebra  $\cO(\bG_m) = k[t,t^{-1}]$ of the linear algebraic group $\bG_m$
with the filtration
 $\{ \cO(\bG_m)_{\leq d} \}$, where $\cO(\bG_m)_{\leq d} \subset \cO(\bG_m$ is the subspace
 spanned by $\{ t^i, \ N\cdot |i| \leq d \}$; this is a subcoalgebra of $\cO(\bG_m)$.

\vskip .1in

We shall utilize the following elementary lemma.  Its proof is a diagram chase involving the square
of maps of (\ref{eqn:subcoalg}) mapping via coproducts to the similar square involving 
tensor squares.

\begin{lemma}
\label{lem:surj-coalg}
Consider a surjective map $\phi: C \twoheadrightarrow C^\prime$ of coalgebras and an injective map  
$j: D \hookrightarrow C$ of coalgebras, and denote by $\ol D$ the subspace 
 $\ol j: (\phi\circ j)(D) \subset C^\prime$.  Then the coproduct $\Delta_D: D \to D\otimes D$ 
 induces a coproduct $\Delta_{\ol D}: \ol D \to
\ol D \otimes \ol D$ which equals the restriction of $\Delta_{C^\prime}: C^\prime \to 
C^\prime \otimes C^\prime$.  In other words, $\phi$ and $j$ determine a commutative square of coalgebras
\begin{equation}
\label{eqn:subcoalg}
 \xymatrixcolsep{5pc}\xymatrix{
 D \ar[d]^\hookrightarrow_j \ar[r]^-{\ol{(-)}}_\twoheadrightarrow  & \ol D \ar[d]^{\ol j}_\hookrightarrow \\
C \ar[r]^\twoheadrightarrow_\phi & C^\prime .
}
\end{equation}
\end{lemma}

\begin{proof}
For $d \in D$, write $\Delta_D(d) \ = \ \sum d_i \otimes d_i^\prime$, so that 
$\Delta_C(j(d)) = \sum j(d_i)) \otimes j(d_i^\prime)$ and 
$$\Delta_{C^\prime}(\phi(j(d))) \ = \ \sum \phi(j(d_i)) \otimes \phi(j(d_i^\prime)) 
\ = \ \sum \ol j (\ol{d_i}) \otimes \ol{j} (\ol{d_i^\prime})\in C^\prime\otimes C^\prime.$$
Thus, $\Delta_{\ol D}(\ol d) \ \equiv \ \sum \ol{d_i} \otimes \ol{d_i^\prime} 
\in \ol D\otimes \ol D$ is well defined since
$\ol j$ is injective.
\end{proof}

\vskip .1in

We apply the above lemma in the special case of the surjective map of coalgebras
$\eta^*: \cO(\bM_{N,N}) \otimes \cO(\bG_m) \ \twoheadrightarrow \ \cO(GL_N)$
and the subcoalgebra \\
$(\cO(\bM_{N,N}) \otimes \cO(\bG_m) )_{\leq d} \ \hookrightarrow \ \cO(\bM_{N,N}) \otimes \cO(\bG_m)$.

\begin{defn}
\label{defn:filt-GL}
Consider the closed immersion of monoid schemes
\begin{equation}
\label{eqn:embed} 
\eta: GL_N \ \hookrightarrow \ \bM_{N,N} \times \bG_m, \quad A \mapsto (A,det(A)^{-1}),
\end{equation}
identifying $GL_N$ as the zero locus of the function $det(\ul x)\otimes t^{-1} \in \cO(\bM_{N,N} \times \bG_m),$
where
\begin{equation}
\label{defn:det}
det(\ul x) \quad = \quad \sum_{\sigma\in \Sigma_N} (-1)^{sgn(\sigma)}\prod_{1\leq i\leq N} x_{i,\sigma(i)}.
\end{equation}
Using Lemma \ref{lem:surj-coalg}, we conclude that
\begin{equation}
\label{eqn:image}
\cO(GL_N)_{\leq d} \quad \equiv \quad \eta^*((\cO(\bM_{N,N}) \otimes \cO(\bG_m) )_{\leq d})
\end{equation}
is a finite dimensional subcoalgebra of $\cO(GL_N)$.
\end{defn}

\vskip .1in

\begin{prop}
\label{prop:O(GLN)}
Let $(\cO(GL_N)_{\leq d})^c$ denote the complement of $\cO(GL_N)_{\leq d-1}$ in $\cO(GL_N)_{\leq d}$.
\begin{enumerate}
\item
For any non-zero element $f \in \cO(GL_N)$, there is uniquely associated $d(f) \geq 0$ 
such that $f \in (\cO(GL_N)_{\leq d(f)})^c$. 
 \item 
 For any $f \in (\cO(GL_N)_{\leq d})^c$, let $e$ be the minimal non-negative integer
such that $f\cdot det(\ul x)^e \in \cO(\bM_{N,N})$.  Then $d(f) = d^\prime +e\cdot N$, where
$d^\prime$ is the minimal non-negative integer such that 
$f\cdot det(\ul x)^e \in \cO(\bM_{N,N})_{\leq d^\prime}$. 
\item
The function sending $d$ to $dim(\cO(GL_N)_{\leq d})$ for a fixed $N$
differs from the function $d \mapsto \frac{d^{N^2}}{(N^2)!}$ by a function 
bounded by a polynomial in $d$ of degree less that $N^2$.
\item
For any $d, e \geq 0$, multiplication in $\cO(GL_N)$ restricts to a map
$$\mu: \cO(GL_N)_{\leq d} \otimes \cO(GL_N)_{\leq e} \quad \to \quad \cO(GL_N)_{\leq d+e}.$$
\end{enumerate}
\end{prop}

\begin{proof}
The proof of assertion (1) follows from the observation that $\{ \cO(\GL_N)_{\leq d} \}$
is an ascending, converging sequence of subspaces (in fact, of subcoalgebras).  
Assertion (2) follows from the fact that $\cO(GL_N)$ is the localization of the 
unique factorization domain $\cO(M_{N,N})$ obtained by inverting $det(\ul x)$.

We identify the underlying vector space of $\cO(\bM_{N\times N})_{\leq d}$ with the space of 
polynomials of total degree $\leq d$ in the polynomial algebra $k[x_{i,j}]$ in $N^2$ variables  
Using induction, one easily verifies that 
\ $dim(\cO(\bM_{N\times N})_{\leq d}) \quad = \quad {d + N^2\choose N^2}$, 
which is a polynomial in $d$ of degree $N^2$ with leading term $((N^2)!)^{-1}$.
Assertion (2) implies that there is a surjective map 
\begin{equation}
\label{eqn:surj-map}
\bigoplus_{i=0}^{[d/N]} \cO(\bM_{N,N})_{\leq d-iN} \quad 
\twoheadrightarrow \quad \cO(GL_N)_{\leq d}
\end{equation}
given by sending $f$ in the summand $\cO(\bM_{N,N})_{\leq d-iN}$ to $f \cdot det(\ul x_{i,j})^{-i}$. 
There is also an evident injective map $\cO(\bM_{N,N})_{\leq d} \hookrightarrow \cO(GL_N)_{\leq d}$,
so that
\begin{equation}
\label{eqn:choose-sum}
{d + N^2\choose N^2} \quad \leq \quad dim(\cO(GL_N)_{\leq d}) \quad \leq \quad
\sum_{i=0}^{[d/N]} {d -iN + N^2\choose N^2}.
\end{equation}
This readily assertion (3).

Observe that multiplication restricts to 
$$(\cO(\bM_{N,N}) \otimes \cO(\bG_m) )_{\leq d} \otimes (\cO(\bM_{N,N}) \otimes \cO(\bG_m) )_{\leq e}
\quad \to \quad (\cO(\bM_{N,N}) \otimes \cO(\bG_m) )_{\leq d+e}.$$  Granted the definition of
$\cO(GL_N)_{\leq d}$ in (\ref{eqn:image}), this immediately implies assertion (4).
\end{proof}

\vskip .1in

Justified by Lemma \ref{lem:surj-coalg}, we introduce the ascending, converging sequence
of subcoalgebras which we shall primarily use.  

\begin{defn}
\label{defn:ascending}
Consider an affine group scheme $\bG$ equipped with a closed embedding 
$\phi: \bG \hookrightarrow GL_N$ for some $N$.
We define the ascending, converging filtration $\{ \cO(\bG)_{\leq d,\phi} \}$ of 
finite dimensional subcoalgebras of $\cO(\bG)$
by setting $\cO(\bG)_{\leq d,\phi}$ equal to $\phi^*(\cO(GL_N)_{\leq d}) \subset \cO(\bG)$.  

For any $d, e \geq 0$, multiplication in $\cO(\bG)$ restricts to a map
\begin{equation}
\label{eqn:mubG}
\mu: \cO(\bG)_{\leq d,\phi} \otimes \cO(\bG)_{\leq e,\phi} \quad \to \quad \cO(\bG)_{\leq d+e,\phi}.
\end{equation}
\end{defn}

\vskip .1in

\begin{ex}
\label{ex:SL}
Consider the closed embedding $\phi: SL_N \hookrightarrow GL_N$ of matrices
of determinant 1.  Then $\cO(SL_N)$ is the quotient of $\cO(GL_N)$ by the principal ideal
$(det(\ul x) - 1)$.  Thus, $det(\ul x), \ (det(\ul x)^{-1} \in \cO(GL_N)_{\leq N}$ both
map to $1 \in \cO(SL_N)_{\leq 0,\phi}$, whereas the coordinate functions 
$x_{i,j} \in  \cO(GL_N)_{\leq 1}$
remain coordinate functions of filtration degree 1 for $\cO(SL_N)$.
\end{ex}

For various unipotent linear algebraic groups $\phi: \bU \hookrightarrow GL_N$, we give a 
familiar description of  $\{ \cO(\bU)_{\leq d,\phi} \}$.

\vskip .1in

\begin{ex} \cite[Ex 2.5]{F18}
\label{ex:unipotent}
Let $\phi: \bU_N \hookrightarrow GL_N$ be the unipotent radical of the Borel subgroup of upper
triangular matrices of $GL_N$.  Then $\phi^*: \cO(GL_N) \ \twoheadrightarrow \ \cO(\bU_N) \simeq k[y_{i,j}; i < j]$
is given by  \ $x_{i,j} \mapsto y_{i,j}, \ i < j$, $x_{i,i} \mapsto 1$, and $x_{i,j} \mapsto 0, \ i > j$.
The coproduct on $\cO(\bU_N) \simeq k[y_{i,j}; i < j]$ is given by
$$\Delta_{\bU_N}(y_{i,j}) \quad = \quad (y_{i,j}\otimes 1) + (\sum_{i < t < j} (y_i\otimes y_t + y_t \otimes y_j)) 
+ (1\otimes y_{i,j}).$$
We identify the subcoalgebra $\cO(\bU_N)_{\leq d,\phi} \subset \cO(\bU_N)$ with 
the subspace of $k[y_{i,j}; i < j]$ consisting of polynomials of total degree $\leq d$ and with coproduct 
the restriction of $\Delta_{\bU_N}$ as above.   Thus,
$$dim(\cO(\bU_N)_{\leq d,\phi}) \quad = \quad \binom{N^\prime+d}{N^\prime}, 
\quad N^\prime = \frac{N^2-N}{2}$$
is a polynomial of degree $N^\prime$ with leading coefficient $1/(N^\prime)!$.

The above discussion for $\cO(\bU_N)_{\leq d,\phi} \subset \cO(\bU_N)$ applies (with minor
modification) to  $\cO(\bU)_{\leq d,\phi} \subset \cO(GL_N)$ whenever $\bU \hookrightarrow GL_N$ 
is the unipotent radical of a parabolic subgroup of $GL_N$ (defined and split over $\bF_p$).
\end{ex}

\vskip .1in

We show that changing the embedding $\phi: \bG \hookrightarrow GL_N$ has limited effect
upon the associated ascending, converging sequences $\{ \cO(\bG)_{\leq d,\phi} \}$.

\begin{prop}
\label{prop:diff-emb}
Let $\bG$ be an affine group scheme  and consider two closed 
embeddings \ $\phi: \bG  \hookrightarrow GL_N, \ \phi^\prime: \bG  \hookrightarrow GL_{N^\prime}.$
There exist positive numbers $c, c^\prime$ such that 
$$\cO(\bG)_{\leq d,\phi} \ \subset  \cO(\bG)_{\leq c\cdot d, \phi^\prime}, \quad
\cO(\bG)_{\leq d,\phi^\prime} \ \subset  \cO(\bG)_{\leq c^\prime \cdot d,\phi}$$
for all $d \geq 0$.
\end{prop}

\begin{proof}  
We define 
\begin{equation}
\label{eqn:c}
c \ = \ min\{ e: \phi^*(\cO(\bM_{N,_N})_{\leq 1}) \subset \cO(\bG)_{\leq e,\phi^\prime}), 
\ \phi^*(det(x_{i,j}))  \in \ \cO(\bG)_{\leq e,\phi^\prime}\}.
\end{equation}
Applying (\ref{eqn:surj-map}), we conclude that $\cO(\bG)_{\leq d,\phi} \ \subset \
\cO(\bG)_{\leq cd,\phi^\prime}$.   We similarly define $c^\prime$ such that
$\cO(\bG)_{\leq d,\phi^\prime} \ \subset  \cO(\bG)_{\leq c^\prime \cdot d,\phi}.$
\end{proof}

\vskip .2in


\section{Filtering $\bG$-modules using subcoalgebras of $\cO(\bG)$}
\label{sec:filt}

We explicitly formulate the filtration by $\bG$-submodules on a $\bG$-module $M$ given
by the ascending, converging sequences $\{ \cO(\bG)_{\leq d,\phi} \}$
of subcoalgebras of $\cO(\bG)$ given in Definition \ref{defn:ascending}.

\begin{defn}
\label{defn:ascending-M}
Consider an affine group scheme $\bG$ equipped with a closed embedding 
$\phi: \bG \hookrightarrow GL_N$ for some $N$.
For any $\bG$-module $M$, we denote by $M_{\leq d,\phi} \subset M$ the largest 
$\cO(\bG)_{\leq d,\phi}$-subcomodule
of $M$.  The ascending, converging sequence  $\{ M_{\leq d,\phi} \}$ of $\bG$-submodules of $M$
equals the filtration of $M$ given in Proposition \ref{prop:exhaust-X} for $\{ X_i \}$ equal to 
$\{ \cO(\bG)_{\leq d,\phi} \}$.  In particular, by Proposition \ref{prop:advantages}(3), 
 $(k[\bG]^R)_{\leq d,\phi}$ equals
$\cO(\bG)_{\leq d,\phi}$ as $\bG$-modules.

If $\phi$ is the identity and $M$ is a $GL_N$-module, then we use $M_{\leq d}$ rather than $M_{\leq d,id}$.
\end{defn}

\vskip .1in

\begin{remark}
\label{rem:angles}
We shall be considering the filtration of Definition \ref{defn:ascending} for this seems somewhat
accessible to computations.  However, one could consider other ascending, converging sequences
of finite dimensional subcoalgebras such as $\{ \cO(\bG)_{\langle X_i \rangle} \}$ determined by 
an ascending, converging sequence of finite dimensional subspaces $\{ X_i \}$ of $\cO(\bG)$.
\end{remark}

\vskip .1in

Recall the Schur algebra $S(N,d)$, the dual of the subcoalgebra  $\cO(\bM_{N,N};d) \subset
\cO(\bM_{N,,N})$ consisting of polynomials in the matrix coefficients $x_{i,j}$ which 
are homogeneous of degree $d$.
A module for $S(N,d)$ (equivalently, a comodule for $\cO(\bM_{N,N};d)$) is called a 
polynomial representation of $GL_N$ homogeneous of degree $d$.

\begin{ex}
\label{ex:Schur}
Let  $M$ be a homogeneous polynomial representation of $GL_N$ of degree $d$.
Then $M_{\leq s}$ equals $0$ if $s < d$ whereas  $M_{\leq s} = M$ if $s \geq d$.

More generally, let $\phi: \bG \ \subset GL_N$ be a closed embedding of a linear 
algebraic group $\bG$ with the property  that \ 
$A(\bG) \ \equiv \ \cO(\bG) \cap \cO(\bM_{N,N})$ can be written as a direct sum 
$\bigoplus_d A(\bG)_d$, where $A(\bG)_d = \cO(\bG) \cap \cO(\bM_{N,N})_d$ (for example, 
the classical orthogonal or symplectic groups).  If $M$ is an object of $CoMod(A(\bG)_d)
\subset Mod(\bG)$, then  $M_{\leq s,\phi} = 0$ if $s < d$ and   
$M_{\leq s, \phi} = M$ if $s \geq d$. See \cite[1.2]{Doty}.
\end{ex}

\vskip .1in

\begin{ex}
Give $\cO(\bG_a) = k[t]$ the evident filtration by degree (equal to that associated to the
embedding of $\phi: \bG_a \hookrightarrow GL_2$ as the unipotent radical of a Borel subgroup).
The subcoalgebra $\cO(\bG_a)_{\leq p^r-1,\phi} \subset \cO(\bG_a)$ is isomorphic as a coalgebra
to the coordinate algebra of $\bG_{a(r)}$; thus, the abelian category $CoMod(\cO(\bG_a)_{\leq p^r-1,\phi})$
is isomorphic to $Mod(\bG_{a(r)})$ which in turn is isomorphic to the
category of modules for the elementary abelian $p$-group $(\bZ/p)^{\times r}$; this category 
is wild if $r > 2$ or if $p > 2$ and $r = 2$.

To give a vector space $M$ the structure of a $\bG_a$-module is equivalent to giving a
sequences of $p$-nilpotent operators $\psi_i: M \to M, \ i \geq 0$ which pair-wise commute
and which satisfy the condition that for each $m \in M$ there exists some $n_m$ such
that $\psi_i(m) = 0, i \geq n_m$.  For a $\bG_a$-module $M$, the $\bG$-submodule
$M_{\leq p^r-1,\phi} \subset M$ consists of those $m \in M$ such that $\psi_i(m) = 0, i \geq r$.
\end{ex}

\vskip .1in

\begin{prop}
\label{prop:tensor-filt}
Let $\bG$ be an affine group scheme equipped with the closed embedding 
$\phi: \bG \hookrightarrow GL_N$ and consider two $\bG$-modules $M, \ M^\prime$.  Then 
$$M_{\leq d,\phi} \otimes 
M^\prime_{\leq d^\prime,\phi} \quad \subset \quad(M\otimes M^\prime)_ {\leq d+d^\prime,\phi}.$$
\end{prop}

\begin{proof}
Since multiplication $\mu: \cO(GL_N) \otimes \cO(GL_N) \to \cO(GL_N)$ restricts to \\
$\cO(\bG)_{\leq d,\phi} \otimes \cO(\bG)_{\leq d^\prime,\phi} \ \to \ \cO(\bG)_{\leq d+d^\prime,\phi}$,
the proposition is a consequence of Proposition \ref{prop:prod-coalg}(1).  
\end{proof}

\vskip .1in

Assuming that $\phi: \bG \hookrightarrow GL_N$ is a linear algebraic group defined over $\bF_{p^r}$
and that the coaction of the $\bG$-module $M$ is also defined over $\bF_{p^r}$,
we next relate $(M^{(r)})_{\leq p^r\cdot d,\phi}$ and $M_{\leq d,\phi}$.  The hypothesis
that $M$ is a $\bG$-module defined over $\bF_{p^r}$ implies that the Frobenius
twist $M^{(r)}$ of $M$ as formulated in \cite[\S 1]{FS} is given by the restriction of 
$M$ along $F^r: \bG \to \bG$ (see \cite[I.9.10]{J}).

\begin{prop}
\label{prop:Frob-twist}
Assume that $\phi: \bG \hookrightarrow GL_N$ is a linear algebraic group defined over $\bF_{p^r}$ and 
that $M$ is a $\bG$-module with
coaction $\Delta_M: M \otimes O(\bG)$ also defined over $\bF_{p^r}$.  
Then the $r$-th Frobenius twist $M^{(r)}$ of $M$ satisfies
\begin{equation}
\label{eqn:r-twist}
M_{\leq d,\phi} \quad = \quad (M^{(r)})_{\leq p^r\cdot d,\phi}.
\end{equation}
\end{prop}

\begin{proof}
Granted our hypotheses, $M^{(r)}$ has coaction
$$\Delta_{M^{(r)}} \ \simeq \  (1_M \otimes (F^r)^*)\circ \Delta_M: M \ \to 
\ M\otimes \cO(\bG) \ \to\ M \otimes \cO(\bG).$$
Thus, (\ref{eqn:r-twist}) follows from the fact that $(F^r)^*: \cO(\bG) \to \cO(\bG)$
sends $f(x_{i,j}) \in \cO(GL_N)$ to $f((x_{i,j})^{p^r})$, multiplying the degree of each monomial
by $p^r$.
\end{proof}

\vskip .1in
We summarize an alternative construction by J. Jantzen  in \cite[Chap A]{J}.
Jantzen's filtration of $\cO(\bG)$ for $\bG$ a reductive algebraic
group has many useful properties, some of which are not satisfied 
by the filtration of Definition \ref{defn:ascending}.   For $\bG = \bG_a$
and $X \subset \cO(\bG_a)$, Remark \ref{rem:not-closed}  
points out that $Mod(\bG_a,X)$ is not closed under extensions and that $R^1(-)_X(k) \not= 0$.

Choose a Borel subgroup 
$\bB \subset \bG$ and denote by $X_+$ denote the set of dominant weights of $\bG$.;   choose a subset
$\pi \subset X_+$.   Jantzen considers the full subcategory $\cC_\pi(\bG) \subset Mod(\bG)$ 
consisting of modules whose objects 
are colimits of finite dimensional $\bG$-submodules with composition factors of the form 
$L(\lambda)$ with $\lambda \in \pi$.   In the special case of $GL_N$ with $\pi \subset X_+$
the set of of dominant weights $\pi(N,d) \subset X_+$ in the notation of \cite[A.3.1]{J}, 
$\cC_{\pi(N,d)}$ is the category of modules for the Schur algebra $S(N,d)$ and $\cO_{\pi(N,d)}(GL_N)$
is the coalgebra $\cO(\bM_{N,N})_d$.

Jantzen considers the functor $\cO_\pi: Mod(\bG) \to \cC_\pi(\bG)$ in a manner similar to the 
construction of $(-)_X: Mod(\bG) \to Mod(\bG,X)$ in Theorem \ref{thm:M-X}.  The condition that
$\pi \subset X_+$ be saturated is the condition that $\mu \in \pi$ whenever there is some 
$\lambda > \mu \in X_+$ with $\lambda \in \pi$.

\begin{thm} Jantzen, \cite[ChapA]{J}
\label{thm:Jantzen}
As above, let $\bG$ be reductive, $\pi\subset X_+$ be saturated, and $\cC_\pi(\bG) \subset Mod(\bG)$. 
Consider the left exactor functor $\cO_\pi(-): Mod(\bG) \to \cC_\pi(\bG)$ right adjoint to the inclusion 
functor.
\begin{enumerate}
\item
$\cO_\pi(\bG) \equiv \cO_\pi(\cO(\bG))$ is a sub-coalgebra of $\cO(\bG)$.  
\item
If $\pi$ is finite, then $\cO_\pi(\bG)$ is finite dimensional (with dimension equal to 
$\sum_{\lambda \in \pi} (dim H^0(\lambda))^2$).
\item
If $\{ \pi_n \}$ is a nested 
sequence of finite, saturated subsets of $X_+$ whose union is all
of $X_+$, then $\{ \cO_{\pi_n}(\bG) \}$ is an ascending, converging sequence of 
finite sub-coalgebras of $\cO(\bG)$.
\item
 The abelian category $\cC_\pi(\bG)$  can be naturally  identified with the abelian category 
 of $\cO_\pi(\bG)$-comodules.  
\item 
$\cO_{\pi(N,d)}(GL_N)$ is the subcoalgebra $\cO(\bM_{N,N})_d  \subset \cO(GL_N)$.
 \item
 $\cC_\pi(\bG)$ is closed under extensions.
 \item
 The higher right derived functors of $\cO_\pi$, $R^i\cO_{\pi}: Mod(\bG) \to \cC_\pi(\bG)$ with $i > 0$,
 vanish on finite dimensional $\bG$-modules in $\cC_\pi$.
\end{enumerate}
\end{thm}

\vskip .1in

We next explore how the  functors 
$$(-)_{\leq d,\phi}: Mod(\bG) \ \to \ CoMod(\cO(\bG)(-)_{\leq d,\phi}, \quad 
(-)_{|\bG_{(r)}}: Mod(\bG) \ \to \ Mod(\bG_{(r)})$$
 complement each other.

\begin{prop}
\label{prop:adjoints}
Let $\bG$ be a linear algebraic group with given embedding $\phi: \bG \to GL_N$,
and let $d, r$ be positive integers.
Then the composition of the natural inclusion with restriction to $\bG_{(r)}$, 
$$ (-)_{|\bG_{(r)}}  \circ i_{\leq d,\phi}: CoMod(\cO(\bG)_{\leq d,\phi}) \quad \hookrightarrow \quad Mod(\bG) 
\quad \to \quad Mod(\bG_{(r)}),$$
is exact and left adjoint to the left exact functor given as the composition 
$$(-)_{\leq d,\phi} \circ ind_{\bG_{(r)}}^\bG(-): Mod(\bG_{(r)})
\quad \to \quad Mod(\bG) \quad \to \quad CoMod(\cO(\bG)_{\leq d,\phi}).$$
\end{prop}

\begin{proof}
The exactness of $(-)_{\bG_{(r)}}$ and of $i_{\leq d,\phi}$ is evident.  Since $\bG/\bG_{(r)}
\ = \ \bG^{(r)}$ is affine, the exactness of  $ind_{\bG_{(r)}}^\bG(-)$ is given by 
\cite{CPS77}.  Thus, the left exactness of   $(-)_{\leq d,\phi} \circ ind_{\bG_{(r)}}^\bG(-)$ is given by 
Theorem \ref{thm:M-X}(3).

The asserted adjunction follows from the adjunction equivalences
$$Hom_{\cO(\bG)_{\leq d,\phi}}(M,(ind_{\bG_{(r)}}^\bG(N))_{\leq d,\phi}) \
\simeq \ Hom_{\bG}(M,ind_{\bG_{(r)}}^\bG(N)) \ \simeq \ 
Hom_{\bG_{(r)}}(M_{|\bG_{(r)}}, N)$$
for any $M \in CoMod(\cO(\bG)_{\leq d,\phi})$ and $N \in Mod(\bG_{(r)})$.
\end{proof}

\vskip .1in

The fact that $(-)_{\leq d,\phi} \circ ind_{\bG_{(r)}}^\bG(-)$ is right adjoint to a 
left exact functor formally implies the following corollary.

\begin{cor}
The functor  
$(-)_{\leq d,\phi} \circ ind_{\bG_{(r)}}^\bG(-): Mod(\bG_{(r)}) \ \to \ CoMod(\cO(\bG)_{\leq d,\phi})$
sends injective/projective $\bG_{(r)}$-modules to injective $\cO(\bG)_{\leq d,\phi}$-comodules. 
\end{cor}

\vskip .1in

We supplement Proposition \ref{prop:adjoints} with another categorical property.

\begin{prop}
\label{prop:full}
Adopt the notation and hypotheses of Proposition \ref{prop:adjoints}.  If
the composition of inclusion and quotient maps of coalgebras, 
$\cO(\bG)_{\leq d,\phi} \ \hookrightarrow \ \cO(\bG) \twoheadrightarrow \cO(\bG_{(r)})$
is an inclusion, then  $ (-)_{|\bG_{(r)}}  \circ i_{\leq d,\phi}: CoMod(\cO(\bG)_{\leq d,\phi})
 \to Mod(\bG_{(r)})$
is a fully faithful embedding of abelian categories.
\end{prop}

\begin{proof}
Consider  two $\cO(\bG)_{\leq d,\phi}$-comodules $M, \  M^\prime$  and 
a $k$-linear map $f: M \to M^\prime$.  This data provides the diagram
\begin{equation}
\label{eqn:extend}
\xymatrix{
M \ar[d]^f \ar[r]^-{\Delta_M} & M\otimes \cO(\bG)_{\leq d,\phi} \ar[r]^-{(-)_{|\bG_{(r)}}} & 
M_{|\bG_{(r)}} \otimes \cO(\bG_{(r)}) \\
M^\prime \ar[r]^-{\Delta_{M^\prime}} & M^\prime \otimes \cO(\bG)_{\leq d,\phi} \ar[r]^-{(-)_{|\bG_{(r)}}} & 
(M^\prime)_{|\bG_{(r)}} \otimes \cO(\bG_{(r)}).
}
\end{equation}
Granted that the composition 
$\cO(\bG)_{\leq d,\phi} \ \hookrightarrow \ \cO(\bG) \twoheadrightarrow \cO(\bG_{(r)})$
is an inclusion, one easily verifies using a simple diagram chase that $f$ is a map of 
$\cO(\bG)_{\leq d,\phi}$-comodules if and only if its restriction to $\bG_{(r)}$
is  a map of $\bG_{(r)}$-modules.
\end{proof}

\vskip .2in


\section{Cofinite $\bG$-modules}
\label{sec:growth}

In this section, we investigate cofinite $\bG$-modules, a class of (necessarily
countable) $\bG$-modules which seem somewhat amenable to study.  We restrict
our attention to linear algebraic groups although the formalism might be useful
for other affine group schemes of countably infinite dimension over $k$.

\vskip .1in

\begin{defn}
\label{defn:cofinite}
Let $\bG$ be a linear algebraic group.   We define a $\bG$-module $M$
to be cofinite if  $M_X$  is finite dimensional for every finite dimensional 
subspace $X \subset \cO(\bG)$.
This condition is equivalent to the condition that each $M_{X_i}$ is a finite dimensional $\bG$-module
for some ascending, converging sequence $\{ X_i \}$ of finite dimensional subspaces of $\cO(\bG)$.

We denote by $CoFin(\bG) \subset Mod(\bG)$ the full subcategory of cofinite $\bG$-modules.
\end{defn}

\vskip .1in

We establish various properties of cofinite $\bG$-modules.   We observe that $k[\bG]^R$ is cofinite
by Proposition \ref{prop:advantages}(3) which asserts that $(k[\bG]^R)_C = C$ for any 
 subcoalgebra $C \subset \cO(\bG)$ containing $1 \in \cO(\bG)$.

 \begin{prop}
 \label{prop:cofinite}
 Let $\bG$ be a linear algebraic group, and let $M, \ E, \ N$ be $\bG$-modules.
 \begin{enumerate}
\item
If $M$ is cofinite, then any $\bG$-submodule of $M$ is also cofinite.
\item 
If $0 \to M \to E \to N \to 0$ is exact and if $M,  \ N$ are cofinite, then $E$ is also cofinite.
\item  
If $M$ is finite dimensional, then $M$ embeds in an injective $\bG$-module which is 
also cofinite.
\item
If either $M$ or $N$ is not cofinite, then $M\otimes N$ is not cofinite.
\item 
If $M$ is a finite dimensional $\bG$-module,
 $M \otimes N$ is cofinite if and only if $N$ is cofinite.

\end{enumerate}  
\end{prop}

\begin{proof}
We recalll that 
the left exactness of $(-)_X$ implies assertions (1) and  (2).  Assertion (3) is justified by 
the natural embedding $M \hookrightarrow M\otimes k[\bG]^R$ together with the observation that
$M\otimes k[\bG]^R$ is isomorphic to $M^{tr}\otimes k[\bG]^R$.  

Choose a closed embedding $\phi: \bG \hookrightarrow \bG$.  To prove assertion (4),
assume that $M_{\leq d,\phi}, \ N_{\leq e,\phi}$  are both non-zero and at least one of them
is infinite dimensional.  Then $(M\otimes N)_{\leq d+e,\phi}$ contains $M_{\leq d,\phi} \otimes N_{\leq e,\phi}$ 
and thus is infinite dimensional.

To prove assertion (5), we show that 
\begin{equation}
\label{eqn:sub}
(M \otimes N)_{\leq d,\phi} \quad \subset \quad M\otimes (N_{\leq \eta(d),\phi})
\end{equation}
for some function $\eta: \bN \to \bN$ (depending upon $\bG$ and the finite
dimensional module $M$).    
Choose a basis $\{ m_i, 1\leq i \leq s \}$ for $M$, a basis $\{ f_\alpha \}$ for $\cO(\bG)$,
a basis $\{ n_\beta \}$ for $N$, and choose $e$ sufficiently large that $M = M_{\leq e,\phi}$.

Consider a simple tensor $m_j \otimes n_{\beta^\prime} \in (M \otimes N)_{\leq d}$ and 
write $\Delta(m_j) = \sum_i m_i \otimes f_{\alpha_{i,j}}$ with $f_{\alpha_{i,j}} \in \cO(\bG)_{\leq e,\phi}$,  \ 
$\Delta(n_{\beta^\prime}) = \sum_\beta n_\beta \otimes f_{\alpha_{\beta,\beta^\prime}}$
where each $f_{\alpha_{i,j}}, f_{\alpha_{\beta,\beta^\prime}}$ is a non-zero multiple of some 
basis element  $f_\alpha$.  The condition that 
 $m_j \otimes n_{\beta^\prime}$ is an element in $(M \otimes N)_{\leq d}$ is equivalent to the condition 
 that each product $f_{\alpha_{i,j}}\cdot f_{\alpha_{\beta,\beta^\prime}}$ is an element of $ \cO(\bG)_{\leq d,\phi}$
 Consequently, a sum of simple tensors $\sum c_{i,\beta} m_i \otimes n_\beta$ is an
element of $(M \otimes N)_{\leq d,\phi}$ if and only if each $m_i \otimes n_\beta$ with
$c_{i,\beta} \not= 0$ is an element of $(M \otimes N)_{\leq d,\phi}$.

Now, consider $f_{\alpha_{i,j}}\cdot (-): \cO(\bG) \to \cO(\bG)$.  Since 
$\cO(\bG)$ is an integral domain, the pre-image $(f_{\alpha_{i,j}}\cdot(-))^{-1}(\cO(\bG)_{\leq d,\phi})$
must be finite dimensional and thus lie in some $\cO(\bG)_{\leq d_{i,j},\phi}$.
Thus, (\ref{eqn:sub}) holds if we take $\eta(d)$ to be the maximum of these $d_{i,j}$.
\end{proof}

\vskip .1in

We observe that a much simpler proof of Proposition \ref{prop:cofinite}(5)
can be given in the case of $GL_N$, establishing the more explicit form of (\ref{eqn:sub}):
\begin{equation}
\label{eqn:sub-GL}
(M \otimes N)_{\leq d} \quad \subset \quad M\otimes (N_{\leq d+e}), \quad \bG = GL_N.
\end{equation}
In this case, every element of $\cO(GL_N)$ is a polynomial in
$det(\ul x)^{-1}$ with coefficients in $\cO(\bM_{N,N})$.  This readily implies that if $f_\beta \notin 
\cO(GL_N)_{\leq d+e}$ then $f_\alpha \cdot f_\beta \notin \cO(GL_N)_{\leq d}$ for any 
$0 \not= f_\alpha \in \cO(GL_N)_{\leq e}$.

\vskip .1in

We caution the reader 
that $M \otimes k[\bG]^R$ is not cofinite whenever $M$ is infinite dimensional
since $(M\otimes k[\bG]^R)_{\cO(\bG)_{\leq 0}} = M^{tr}$. 

A more pervasive ``failing" of the full subcategory $CoFin(\bG) \ \hookrightarrow Mod(\bG)$
of cofinite $\bG$-modules is that the quotient of a cofinite $\bG$-module need not be cofinite.
Thus, $CofFin(\bG)$ is not an abelian subcategory of $Mod(\bG)$ even for $\bG = \bG_a$.

\begin{ex}
\label{ex:not-cofinite-Ga}
Consider the $\bG_a$-submodule $Q \subset k[t]^R$ spanned by $\{ t^{p^i}, i \geq 0 \}$.
Then $Q$ is cofinite (see Proposition \ref{prop:socle} below) with $soc(Q) = k$.  However,
the quotient $Q/soc(Q)$ is the trivial $\bG_a$-module spanned by the images of $\{ t^{p^i}, i > 0 \}$.

In contrast, if $\bG$ is reductive and if $M$ is the direct sum 
$\bigoplus_{i > 0} S_{p^i\cdot \lambda}$
for some dominant (positive) weight $\lambda$, then $M$ is cofinite but does not have
finite dimensional socle.
\end{ex}

\vskip .1in

\begin{prop}
\label{prop:socle}
For any linear algebraic group $\bG$ and any $\bG$-module $M$, if $soc(M)$ is finite 
dimensional then $M$ is cofinite.

If $\bU$ is a unipotent linear algebraic group, then a $\bU$-module $M$ is cofinite if and only 
if $soc(M) \ = \ M^{\bU}$ is finite dimensional.
\end{prop}

\begin{proof}
 Observe that any $\bG$-module $M$ admits a
$\bG$-equivariant embedding into $M^{tr} \otimes k[\bG]^R$ (where $M^{tr}$ is the underlying 
vector space of $M$ with the trivial $\bG$-action).  In particular, there is a $\bG$-equivariant
embedding $soc(M) \hookrightarrow soc(M) \otimes k[\bG]^R$ which extends by the
injectivity of $soc(M) \otimes k[\bG]^R$ to a map $j_M: M \to soc(M) \otimes k[\bG]^R$.  
By Proposition \ref{prop:cofinite}(5), $soc(M) \otimes k[\bG]^R$ is cofinite.
Since any irreducible submodule of the kernel of $j_M$ must be contained in $soc(M)$
and since $j_M$ restricts to an injection on $soc(M)$, we 
conclude that $j_M$ is an embedding.  Thus, $M$ is cofinite by Proposition \ref{prop:cofinite}(1).

If $\bU$ is unipotent and $M$ is any $\bU$-module, $M^\bU = soc(M) = M_{\leq 0,\phi}$.  Thus, 
if $soc(M)$ is not finite dimensional, the $M$ is not cofinite.
\end{proof}

\vskip .1in

The invariant $\gamma(\bG,\phi)_M$ of a cofinite $\bG$-module given in (\ref{eqn:poly-growth}) below
 is only one of many similar invariants one might define.

\begin{defn}
\label{defn:poly-growth}
Let $\bG$ be a linear algebraic group equipped with a closed embedding $\phi: \bG \hookrightarrow GL_N$.
 We say that a cofinite $\bG$-module $M$ has 
cofinite type $\gamma(\bG,\phi)_M$ equal to $(e,c)$ if
\begin{equation}
\label{eqn:poly-growth}
 \varinjlim_d \frac{dim(M_{\cO(\bG)_{\leq d,\phi}})}{d^e} \quad = \quad c \ > \ 0.
 \end{equation}
 For such a $\bG$-module $M$, we say that $M$ has polynomial growth of degree $e$ 
 with leading coefficient $c$.
 \end{defn}

\vskip .1in

\begin{ex}
\label{ex:cofinite}
Proposition \ref{prop:O(GLN)}(3) tells us that the $\cO(GL_N)$-module 
 $k[GL_N]^R$ has cofinite type $(N^2,((N^2)!)^{-1})$.
 
 Example \ref{ex:unipotent}
tells us that the $\cO(\bU_N)$-module $k[\bU_N]^R$ has cofinite type \\
 $(N^\prime,((N^\prime)!)^{-1})$ where $N^\prime = \frac{N^2-N}{2}$.

The cofinite type of the mock injective $\bG_a$-module $J_d$ of Example \ref{ex:Ga-mock} 
equals $(1,p^{-d})$ (where $\phi: \bG_a \hookrightarrow GL_2$ is the closed embedding
of strictly upper triangular matrices).
\end{ex}

\vskip .1in

\begin{ex}
\label{ex:not-poly}
The $\bG_a$-submodule $P \ \equiv \{1, t^{p^i}\} \subset k[\bG_a]^R$ 
of primitive elements
satisfies $dim(P_{\cO(\bG_a)_{\leq p^r}}) \ = \ r+1$, so that one could say $P$ has logarithmic growth.

Let $V$ be the natural representation of $GL_N$ of dimension $N$, a 
polynomial representation homogeneous of degree $1$ (see Example \ref{ex:GL-poly}).  Set
\ $M \ \equiv \ \bigoplus_{n \geq 0} (V^{(n)})^{\oplus n!}.$ \  Then 
$$dim(M_{\leq p^r} ) \ = \  \sum_{0 \leq s \leq r} N\cdot s! ,$$
so that $d \mapsto dim(M_{\leq d})$ grows faster than any polynomial in $d$. 
\end{ex}

\vskip .1in

\begin{ex}
\label{ex:GL-poly}
Let $P$ be a polynomial representation of $GL_N$ of dimension $n$ which is homogeneous 
of degree $s$ and let $M = S^*(P)$ be the symmetric algebra on $P$ 
viewed as a $GL_N$-module.  Since the coaction of $\cO(GL_N)$ on $P$ factors
through $\cO(\bM_{N,N})$, \ $M$ is a graded $\cO(\bM_{N,N})$-module with 
$M_{\leq d\cdot s} \  = \ S^d(P)$.
Thus, $dim(M_{\leq d}) \ = \ \binom{([d/s]+n}{n}$ (where $[d/s]$ is the largest integer $\leq d/s$), which as
a function of $d$ 
differs from $d \mapsto \frac{d^{n}}{s^{n}\cdot n!}$ by an error term of degree (in $d$)
less than $n-1$.

Thus, $M$ is a cofinite $GL_N$-module with $\gamma(GL_N)_P = (n,(s^n\cdot n!)^{-1})$.
\end{ex}

\vskip .1in

As we see below, the polynomial growth of a cofinite $\bG$-module $M$ is independent 
of the choice of closed embedding $\phi:\bG \hookrightarrow GL_N$.

\begin{prop}
\label{prop:equal-growth}
Let $\bG$ be a linear algebraic group and $M$ a cofinite $\bG$-module.  Consider
two closed embeddings $\phi: \bG \hookrightarrow GL_N
\ \phi^\prime: \bG \hookrightarrow GL_{N^\prime}$.  If  $\gamma(\bG,\phi)_M \ = \ (e,c)$ and  
if $\gamma(\bG,\phi^\prime)_M \ = \ (e^\prime,c^\prime)$,
 then $e \ = \ e^\prime$.
 
 In particular (see Definition \ref{defn:ascending-M}), we conclude that the polynomial 
 growth of $d \mapsto \cO(\bG)_{\leq d,\phi}$ is independent of the embedding $\phi$.
 \end{prop}
 
 \begin{proof}
 If $\phi, \psi: \bN \to \bN$ are sequences of polynomial growth $e,f$ 
respectively, then $\phi\circ\psi$ has polynomial growth $e\cdot f$.
In particular, given an ascending, converging sequence $n \mapsto \phi(n)$
of polynomial growth $e$, then a subsequence $n \mapsto \phi (\psi(n))$ with
$\psi(n)$ growing linearly in $n$ also has growth $e$

Thus, the proposition follows by appealing to Proposition \ref{prop:diff-emb}. 
 \end{proof}

\vskip .1in

We next compute the degree of polynomial growth of the right reqular representation $k[\bG]^R$
and $\bG$-submodules of the form $k[\bG/\bH]^R \subset k[\bG]^R$.
 
\begin{prop}
\label{prop:growthbG}
Let $\phi: \bG \hookrightarrow GL_N$ be a smooth, closed embedding of linear algebraic groups, 
and let $\fg$ denote the Lie algebra of $\bG$.  Then the polynomial growth of 
\ $d \mapsto dim(\cO(\bG)_{\leq d,\phi})$ is equal \ $dim(\fg)$.

Moreover, consider a closed embedding $\bH \hookrightarrow \bG$ of linear algebraic groups 
with $\bG/\bH$ affine and view $\cO(\bG/\bH)$ as the $\bG$-submodule $k[\bG/\bH]^R$
of $k[\bG]^R$.  Then the polynomial growth of
\ $d \mapsto dim(\cO(\bG/\bH)_{\leq d,\phi})$ is equal to \ $dim(\fg) - dim(\fh)$.
\end{prop}

\begin{proof}
Let $n$ denote the dimension of the smooth variety 
$\bG$ and let $I \subset \cO(\bG)$ denote the augmentation ideal of $\cO(\bG)$,
equal to the maximal ideal at the identity of $\cO(\bG)$.   Since $\cO(\bG)_{I}$ is a regular local
ring of dimension $n$, the dimension of $\cO(\bG)/I^{d+1}$ equals $\binom{n+d}{n}$; thus,
$d \mapsto \cO(\bG)/I^{d+1}$ has polynomial growth of degree $n$ (in the sense of Definition
\ref{defn:poly-growth}).  

Choose $e$ such that the composition 
$\cO(\bG)_{\leq e,\phi} \hookrightarrow \cO(\bG) \twoheadrightarrow
\cO(\bG)/I^2$ is surjective and choose $j$ such that 
$\cO(\bG)_{\leq 1,\phi} \to \cO(\bG)/I^j$ is injective.  Then
\begin{equation}
\label{eqn:inequal}
dim(\cO(\bG)_{\leq d,\phi}) \quad \leq \quad dim(\cO(\bG)/I^{d\cdot j})
\quad \leq \quad dim(\cO(\bG)_{\leq e \cdot d\cdot j}).
\end{equation}
Thus, as in the proof of Proposition \ref{prop:equal-growth}, $d \mapsto dim(\cO(\bG)_{\leq d,\phi})$
has polynomial growth equal to $n$.

Let $\cM \subset \cO(\bG/\bH)$ denote the maximal ideal at the coset $[\bH] \in \bG/\bH$.
Since $\bG/\bH$ is smooth of dimension equal to \ $dim(\fg) - dim(\fh)$, the local ring
$(\cO(\bG/\bH))_{\cM}$ is a regular local ring of the same dimension.  Consequently,
$d \mapsto dim(\cO(\bG/\bH)/\cM^{d+1})$ has polynomial growth equal to $dim(\fg) - dim(\fh)$.
Choose $e$ such that the composition 
$\cO(\bG/\bH)_{\leq e,\phi} \ = \ \cO(\bG/\bH) \cap \ k[\bG]^R_{\leq e,\phi} \ \to \
\cO(\bG/\bH)/\cM^2$ is surjective, and choose $j$ such that 
$\cO(\bG/\bH)_{\leq 1,\phi} \ \to \ \cO(\bG/\bH)/\cM^j$ is injective.  Then 
\begin{equation}
\label{eqn:inequal2}
dim(\cO(\bG/\bH)_{\leq d,\phi}) \quad \leq \quad dim(\cO(\bG/\bH)/\cM^{d\cdot j})
\quad \leq \quad dim(\cO(\bG/\bH)_{\leq e \cdot d\cdot j}),
\end{equation}
so that argument of the proof of Proposition \ref{prop:diff-emb} implies that 
$d \mapsto dim(\cO(\bG/\bH)_{\leq d,\phi})$
has polynomial growth equal to $dim(\fg)${\text -}$dim(\fh)$.
\end{proof}
 
 \vskip .1in
 
 As a consequence of Proposition \ref{prop:growthbG}, we obtain the following.

\begin{prop}
\label{prop:left-right}
Let $\phi: \bG \hookrightarrow GL_N$ be a smooth, closed embedding of affine group schemes, 
and let $\fg$ denote the Lie algebra of $\bG$.
The $\bG$-modules 
$$k[\bG]^R, \quad k[\bG]^L, \quad k[\bG]^{Ad}$$
 each have polynomial
growth of degree equal to $dim(\fg)$ with respect to $\{ \cO(\bG)_{\leq d,\phi} \}$.
\end{prop}

\begin{proof}
The growth of $k[\bG]^R$ with respect to $\{ \cO(\bG)_{\leq d,\phi} \}$ equals $dim(\fg)$ by
Proposition \ref{prop:growthbG}.  Since $\sigma_\bG: k[\bG]^L \to k[\bG]^R$ is an isomorphism
of $\bG$-modules, we conclude that the
growth of $k[\bG]^L$ with respect to $\{ \cO(\bG)_{\leq d,\phi} \}$ also equals $dim(\fg)$.

We consider the effect of the antipode $\sigma_{GL_N}$ on filtrations.
For any invertible $N\times N$-matrix $A$, Cramer's rule tell us that $A$
has inverse $B = (b_{i,j})$, where $b_{i,j} = (-1)^{i+j} det(A_{j,i})\cdot det(A)^{-1}$ where
$A_{j,i}$ is the $N{\text -}1\times N{\text -}1$ matrix obtained by eliminating the 
$j$-th row and $i$-th column of $A$.  
Thus, $\sigma_{GL_N}(x_{i,j}) \in \cO(GL_N)$ is the product of the $N{\text -}1$ degree
polynomial $(-1)^{i+j}det(\{ x_{s,t}, s \not=j, t\not= i \})$ and the function $det(\ul x)^{-1}$ 
(which is given filtration degree $N$); in other words, $\sigma_{GL_N}(x_{i,j}) \in \cO(GL_N)_{\leq 2N-1}$
but $\sigma_{GL_N}(x_{i,j}) \notin \cO(GL_N)_{\leq 2N-2}$.  Since $\sigma_{GL_N}: \cO(GL_N)
\to \cO(GL_N)$ is an anti-algebra morphism, we conclude that $\sigma_{GL_N}$ restricts to
$\sigma_{GL_N}: \cO(GL_N)_{\leq d} \to \cO(GL_N)_{\leq (2N-1)d}$.

The coaction determining the comodule structure of $k[GL_N]^{Ad}$ is the composition 
$$\Delta_{Ad} \ \equiv \ \tau \circ \mu_{1,3} \circ (\sigma_{GL_N}\otimes 1 \otimes 1) 
\circ (1\circ \Delta_{GL_N}) \circ \Delta_{GL_N}:$$
$$\cO(GL_N) \ \to \ \cO(GL_N)^{\otimes 2} \ \to \ \cO(GL_N)^{\otimes 3} \ \to \ \cO(GL_N)^{\otimes 3} \ \to \ 
\cO(GL_N)^{\otimes 2} \ \to \ \cO(GL_N)^{\otimes 2},$$
where $\mu_{1,3}$ multiplies the first and third tensor factors.  (See \cite[I.2.8(7]{J}.)
Writing $\Delta_{Ad}(x_{i,j})$ as  $\sum_{s,t} x_{s,t} \otimes f_{s,t}^{i,j}$,
we observe that each  $f_{s,t}^{i,j}$ is a product of a function of filtration degree $\leq 2N-1$
and a function of degree 1, thus $f_{s,t}^{i,j}$ has filtration degree $\leq 2N$.  
Thus $x_{i,j} \in (k[GL_N]^{Ad})_{2N}$.
We conclude that 
$$dim(\cO(GL_N)_{\leq d}) \quad \leq \quad dim( (k[GL_N]^{Ad})_{2Nd}) \quad
 \leq \quad dim(\cO(GL_N)_{\leq 2Nd}),$$ 
so that once again the argument of the proof of Proposition \ref{prop:diff-emb} implies that the
polynomial growth of $k[GL_N]^{Ad}$ is also $N^2$.

Observe that 
$\sigma_\bG(\phi^*(x_{i,j})) \in \cO(\bG)$ has filtration degree $\leq 2N-1$.
Since the restriction $\phi^*: \cO(GL_N) \to \cO(\bG)$ is a surjective map of Hopf algebras 
and since $\cO(\bG)_{\leq d,\phi}$ is defined to be $\phi^*(\cO(\bG)_{\leq d})$, we may 
apply $\phi^*$ to the above arguments
for $GL_N$ to conclude the corresponding statements for $\bG$.  
\end{proof}

\vskip .2in


\section{Mock injective $\bG$-modules  }
\label{sec:mock}

A $\bG$-module $M$ for a (connected) linear algebraic group is called mock injective
if the restriction $M_{|\bG_{(r)}}$ of $M$ to each Frobenius kernel $\bG_{(r)}$ is an injective $\bG_{(r)}$-module.
Every injective $\bG$-module is mock injective.

The following list of properties of mock injective $\bG$-modules following easily from the 
exactness of $(-)_{|\bG_{(r)}}: Mod(\bG) \to Mod(\bG_{(r)})$
and the corresponding properties for support properties for $\bG_{(r)}$-modules.  (See \cite[Prop 4.6]{F18}.)

\begin{prop}
\label{prop:property-mock}
Let $\bG$ be a  linear algebraic group.  
\begin{enumerate}
\item
A $\bG$-module is mock injective if and only if its support variety $\Pi(G)_M$ 
(as defined in \cite{F23}) is empty.
\item
A directed colimit $\varinjlim_i M_i$ of mock injective $\bG$-modules is mock injective.
\item
Let $0 \to M_1 \to M_2 \to M_3 \to 0$ be an exact sequence of $\bG$-modules.  If
two of $ M_1, M_2, M_3 $ are mock injective, then the third is also mock injective.
\item
If $\mathbb H \hookrightarrow \mathbb G$ is a closed embedding of linear algebraic
groups and $M$ is a mock injective $\bG$-module, then the restriction to $\bH$ of
$M$ is a mock injective $\bH$-module.
\end{enumerate}
None of the above properties is valid
for all linear algebraic groups if ``mock injective" is replaced by ``injective."

(5)  As for injective $\bG$-modules, if $M$ is a mock injective $\bG$-module
then $M \otimes N$ is also mock injective for all $\bG$-modules $N$.
\end{prop}

\vskip .1in

We continue to find the class of mock injective $\bG$-modules mysterious.  The dichotomy
between the classes of injective $\bG$-modules and mock injective $\bG$-modules 
is emphasized by the following proposition.

\begin{prop}
\label{prop:dichotomy}
Let $\bG$ be a linear algebraic group with closed embedding $\phi: \bG \hookrightarrow GL_N$.
\begin{enumerate}
\item
$M$ is mock injective if and only if the support variety
$\Pi(\bG_{(r)})_{M_{|\bG_{(r)}}}$ is empty for all $r > 0$.
\item
$M$ is injective if and only if the $\bG$-module $M_{\leq d,\phi}$ 
is an injective $\cO(\bG)_{\leq d,\phi}$ comodule for every $d \geq 0$.
\end{enumerate}
\end{prop}

\begin{proof}
Assertion (1) follows from the detection of injectivity property for support varieties for infinitesimal group
schemes (\cite{SFB2}, \cite{Pevt}) now expressed in the notation/terminology of \cite{F23}.

Assertion (2) is given by Proposition \ref{prop:X-injectivity}.
\end{proof}

\vskip .1in

If a mock injective $\bG$-module is not injective (as a $\bG$-module), then it is called
a proper mock injective $\bG$-module. A $\bG$-module $L$ is an injective 
if and only if every short exact sequence $0 \to L \to M \to N \to 0$
of $\bG$-modules splits.  If $J$ is a proper mock injective $\bG$-module $J$ with a embedding
$J \hookrightarrow I$ of $J$ into an injective $\bG$-module $I$, then the short exact sequence
$0 \to J \to I \to I/J \to 0$ does not split.

\begin{remark}
\label{rem:parshall}
Mock injective modules are necessarily infinite dimensional.  They contrast greatly to ``Parshall's Conjecture"
proved by Lin and Nakano \cite[Cor 3.5]{L-N} which states that for a finite dimensional $\bG$-module $M$
that if $M$ is injective over $\bG_{(1)}$ then it must be projective over $\bG(\bF_p)$.  
See Example \ref{ex:Ga-mock} below and extended to the analogue for the restriction to 
$\bG_{(r)}$ versus restriction to $\bG_{\bF_{p^r}}$ in \cite[Thm 4.5, Prop 5.1]{F11}.
\end{remark}

\vskip .1in

We recall the first examples of proper mock injective 
$\bG$-modules, an interpretation of
results of Cline, Parshall, and Scott concerning induced modules. (See \cite{CPS77}.)

\begin{ex} \cite[Prop 4.54]{F18}
\label{ex:first-mock}
Let $\bG$ be a linear algebraic group and $j: \bH \hookrightarrow \bG$ a closed embedding
of the linear algebraic group $\bH$.
Then the restriction $(k[\bG]^R)_{|\bH}$ to $\bH$ of the right regular representation of $\bG$ is
a mock injective $\bH$-module.  On the other hand, $(k[\bG]^R)_{|\bH}$ is an injective $\bH$-module
if and only if $\bG/\bH$ is an affine variety.

In particular, if $\bG$ is a reductive algebraic group and $\bH$ is not reductive, then
$(k[\bG]^R)_{|\bH}$ is a proper mock injective $\bH$-module.
\end{ex}

\vskip .1in 

\begin{remark}
\label{rem:useful}
The reader may gain some intuition, as we have, by comparing the fixed points of $k[\bG_a]^R$
under the actions of the finite subgroup schemes $\bG_a(\bF_p)$ and $\bG_{a(1)}$ of $\bG_a$.  
We identify these actions using
$$ (k[\bG_a(\bF_q)])^* \otimes k[\bG_a(\bF_q)]^R \ \to \ k[\bG_a(\bF_q)]^R, \quad
\alpha \otimes t^n  \ \mapsto \ \sum_{i\leq n} {n \choose i} \alpha^{n-i} t^i$$
$$(k[\bG_{a(r)}])^* \otimes k[\bG_a(\bF_q)]^R \ \to \ k[\bG_a(\bF_q)]^R, \quad
\hat{t^i} \otimes t^n \ \mapsto \  \sum_{i\leq n}  {n \choose i} t^{n-i};$$
here, $\alpha$ denotes an element of $\bG_a(\bF_q)$, a generator of $(k[\bG_a(\bF_q)])^*$.
The fixed point space $k[t]^{\bG(\bF_p)}$ consists of those $f(t)$ which are polynomials
in $t^p$-$t$ (see Proposition \ref{prop:ind}); this is a proper mock injective module by
Theorem \ref{thm:HNS}.  On the other hand, the fixed point 
space $k[t]^{\bG_{a(1)}}$ consists of those $f(t)$ which are polynomials in $t^p$; the restriction
to $\bG_{a(1)}$ of this fixed point space has trivial $\bG_{a(1)}$-action and thus is not 
an injective $\bG_{a(1)}$-module.

See Example \ref{ex:Ga-mock} for an investigation of the role of $\bG_a(\bF_q)$ fixed points
in constructing mock injective $\bG_a$-modules.
\end{remark}

\vskip .1in

We provide a somewhat surprising property of mock injective $\bG$-modules.

\begin{thm}
\label{thm:mock-quotients}
Consider a linear algebraic group $\bG$, a finite dimensional subspace $X \subset \cO(\bG)$,
and a mock injective $\bG$-module $J$.  Then 
$$Hom_\bG(J,M_X) \quad = \quad 0$$
for every $\bG$-module $M$ provided that either
\vskip .05in

(i.)  $\bG$ is unipotent,  \quad or 
\vskip .05in

(ii.)  $\bG$ is semi-simple and simply connected and $p \geq 2h-2$, where $h$ is
the Coxeter number of $\bG$. 

\end{thm}

\begin{proof}
First, consider a unipotent linear algebraic group $\bU$ with
closed embedding $\phi: \bU \hookrightarrow GL_N$ and a mock injective $\bU$-module
$J$.  Assume the existence of a map $\rho: J \to M$ sending $j\in J$ to $m \not= 0 \in M$.
For all $r > 0$, $J_{|\bU_{(r)}}$ is a free $\bU_{(r)}$-module so that $j$ lies in 
some free, rank 1 $\bU_{(r)}$-summand $(J_{|\bU_{(r)}})_j$ of $J_{|\bU_{(r)}}$.  Namely,
we can take the cyclic $\bU_{(r)}$ submodule $k\bU_{(r)}\cdot j \subset J_{|\bU_{(r)}}$,
and extend this inclusion to the injective hull of $k\bU_{(r)}\cdot j$ which is a free, rank 1
$\bU_{(r)}$-summand of $J_{|\bU_{(r)}}$.    For notational simplicity, let $J_j$ denote 
$(J_{|\bU_{(r)}})_j$.

Assume that $M = M_{\leq d,\phi}$
and choose $r$ sufficiently large that the composition $\cO(\bU)_{\leq d,\phi} \subset \cO(\bU)
\twoheadrightarrow \cO(\bU_{(r-1})$ is injective.  Choose an identification  $\cO(\bU_{r)}) \ \simeq  
 J_j$ of $\bU_{(r)}$-modules and choose $\psi \in k\bU_{(r)} = (\cO(\bU_{r)}))^*$
such that $\psi$ vanishes on 
the image $X$ of $\cO(\bU)_{\leq d,\phi} \subset \cO(\bU) \twoheadrightarrow \cO(\bU_{(r)})$
but does not vanish on a generator $g_r$ of $J_j$.  Observe that 
$\Delta_{J_j}(g_r) \notin J_j\otimes X$.  This implies that $\rho(g_r) \notin M_{\leq d,\phi}$
and thus must equal 0.
Since $g_r$ is a generator of $J_j$, there exists some $\theta  \in k\bU_{(r)}$ sending $g_r$ to $j$,
so that $\theta(\rho(g_r)) = \rho(\theta(g_r)) = \rho(j) = m$.  This contradicts our
assumption that $0 \not= m \in M$.

Assume now that $\bG$ is semi-simple and simple connected and that $M = M_X$ 
for some finite dimensional subspace $X \subset \cO(\bG)$. Once again we proceed by 
contradiction, assuming the exists a map $\rho: J \to M$ sending $j\in J$ to $ m \not= 0 \in M$.
Let $\{ S_\lambda \}$
be the set of isomorphism classes of irreducible $\bG$-modules, each of finite 
dimension $d(\lambda)$, and assume that the injective hull of $S_\lambda$ restricted to $\bG_{(r)}$
for $r$ sufficiently large admits a $\bG$-structure.  This implies that the injective hull of
$S_\lambda$ as a $\bG$-module is the colimit the injective hulls as $\bG_{(r)}$-modules
(see \cite[II.11.17]{J}).  Then $k[\bG]$ is isomorphic as a $\bG$-module to 
a direct sum of $I_\lambda$'s, with $I_\lambda$ occurring finitely often depending on $d_\lambda$.
Each $I_\lambda$ is infinite dimensional (see \cite[II.11.4]{J}).  

Since $J$ embeds into 
$J^{tr} \otimes k[\bG]^R$ as a $\bG$-module and since this embedding splits when restricted
to each $\bG_{(r)}$, we conclude that $J_{|\bG_{(r)}}$ for each $r > 0$ is a direct sum of
cyclic modules, the images of generators of copies of $\cO(\bG_{r)})$
constituting $(J^{tr} \otimes k[\bG]^R)_{\bG_{(r)}}$.   Moreover, as $r$ increases, the
dimension of these cyclic summands of $J_{|\bG_{(r)}}$ also increases.  Thus, if the restriction
 to $J_{|\bG_{(r)}}$ of $g_r \in J$ is the image of a generator of a copy of $\cO(\bG_{r)})$ 
 in $\cO(\bG)_{|\bG_{(r)}}$, then $\Delta_J(g_r) \in  J\otimes \cO(\bG)$ 
 restricts to a sum of terms in $J_{\bG_{(r)}} \otimes \cO(\bG_{(r)})$ such that the number of
 terms of this sum grows as $r$ increases.

We can thus repeat the argument given for $\bU$, now replacing
``some free $\bU_{(r)}$-summand $(J_{|\bU_{(r)}})_j$ $ \ \ldots \  \ldots g_r \in (J_{|\bU_{(r)}})_j$"
by a cyclic summand of  $(J_{|\bU_{(r)}})_j$ of $J_{|\bG_{(r)}}$ containing the chosen element $j$
assumed to map to $m \not= 0 \in M$ and also containing some $g_r$ which is the image of
a generator of a copy of $\cO(\bG_{(r)})$  in $\cO(\bG)_{|\bG_{(r)}}$.  Namely, these conditions
enable us to chose $\psi \in k\bG_{(r)} = (\cO(\bG_{r)}))^*$
such that $\psi$ vanishes on 
the image $X$ of $\cO(\bG)_{\leq d,\phi} \subset \cO(\bG) \twoheadrightarrow \cO(\bG_{(r)})$
but does not vanish on $g_r$ of $J_j \subset J$.  Once again, we obtain a contradiction 
because $j \in J$ when restricted to $J_j$ is the image under some $\theta \in k\bG_{(r)}$ of $g_r$.
\end{proof}

\vskip .1in

We summarize some useful properties of the functor $ind^{\bG}_H(-)$ following
\cite{CPS77}

\begin{prop}
\label{prop:ind-H}
Let $\bG$ be a linear algebraic group and $H \subset \bG$ a closed subgroup scheme with the property
that $\bG/H$ is affine.
\begin{enumerate}
\item
The induction functor \ $ind^{\bG}_H(-): Mod(H) \ \to \ Mod(\bG)$ \
is an exact functor between abelian categories 
\item
If $M$ is an injective $H$-module, then $ind_H^\bG(M)$ is an injective $\bG$-module.
\item
There is a natural identification $H^i(\bG,ind_H^\bG(M)) \simeq H^i(H,M)$ for any
$H$-module $M$, any $i \geq 0$.
\item
If $M$ is an $H$-module such that $H^n(H,k) \not=0$ for some $n > 0$, then
$Ind_H^\bG(M)$ is not an injective $\bG$-module.
\item
If $M$ is a finite dimensional $H$-module, then $ind_H^\bG(M)$ is a cofinite $\bG$-module.
\end{enumerate}
\end{prop}

\begin{proof}
The exactness of $ind^{\bG}_H(-)$ is proved in \cite{CPS77}.  
Assertion (2) follows from the 
fact that $ind^{\bG}_H(-)$ has an exact left adjoint (namely, the restriction functor) (see the
proof of Corollary \ref{cor:enough}).  The equivalence $H^0(\bG,ind_H^\bG(M)) \simeq H^0(H,M)$
is given by the universal property for $Hom_\bG(-,ind_H^\bG(M))$ applied to the $\bG$-module $k$.
Assertion (3) follows by applying the Grothendieck spectral sequence for the composite functor
$H^0(\bG,-) \circ ind_H^\bG(-)$ together with this equivalence and the exactness of $ind_H^\bG(-)$.
Assertion (3) immediately implies assertion (4).

Finally, if $M$ is finite dimensional, then we may realize $ind_H^\bG(M)$ as the fixed point space
under the right action of $H$ on $M^{tr} \otimes k[\bG]_L$ and thus a $\bG$-submodule of $M^{tr} \otimes k[\bG]_L$.
Hence, assertion (5) follows from the fact that $M^{tr} \otimes k[\bG]_L$ is a cofinite $\bG$-module.
\end{proof}

\vskip .1in

We supplement Proposition \ref{prop:ind-H} with the following proposition concerning the
composition $ind_H^\bG(-) \circ (-)_H$.

\begin{prop}
\label{prop:ind-restrict}
Let $\bG$ be a linear algebraic group and $H \subset \bG$ a closed subgroup scheme with the property
that $\bG/H$ is affine.  Then the exact functor
$$ind_H^\bG(-) \circ (-)_H: Mod(\bG) \quad \to \quad Mod(\bG)$$
is faithful.

Moreover, if $M, \ N$ are $\bG$-modules and if $\psi: ind_H^\bG(M_{|H}) \to ind_H^\bG(N_{|H})$ is
an isomorphism of $\bG$-modules, then $\psi$ induces an isomorphism of $H$-modules
$M_{|H} \stackrel{\sim}{\to} N_{|H}$.
\end{prop}

\begin{proof}
To prove that $ind^{\bG}_H(-)\circ (-)_H$ is faithful, we observe for $\bG$-modules $N$ and $M$ 
that if $f: N_{|H} \to M_{|H}$ is a map of 
$H$-modules, then $ind^{\bG}_H(f)$: $ind^{\bG}_H(N_{|H}) 
\to ind^{\bG}_H(M_{|H})$ has the property
that pre-composition with the canonical $\bG$-map $i_N: N \to ind^{\bG}_H(N_{|H})$ and 
post-composition with the canonical $H$-map 
$ev_H: ind^{\bG}_H(M_{|H})\\ \twoheadrightarrow M_{|H}$ equals $f: N_{|H} \to M_{|H}$.

Let $N, \ M$ be $\bG$-modules and let $\psi: ind^{\bG}_H(N_{|H}) \to ind^{\bG}_H(M_{|H})$ be 
an isomorphism of $\bG$-modules.  Observe that the identity map $N \to N_{|H}$ determines
$i_N: N \to ind^{\bG}_H(N_{|H})$ which is inverse to the canonical map $ev_N: ind^{\bG}_H(N_{|H}) \to
N_{|H}$ of $H$-modules.  The naturality of this splitting implies that the map
$ev_M \circ \psi \circ i_N: N \to M_H$ of $H$-modules when composed with $ev_N \circ \psi^{-1} \circ i_M:
M \to N_H$ equals the identity of $N_H$.  This implies that $ev_M \circ \psi \circ i_N$ is an
isomorphism of $H$-modules.
\end{proof}

\vskip .1in

In  \cite[\S2]{HNS}, 
Hardesty, Nakano, and Sobaje provide a method for constructing many proper mock
injective modules.  We elaborate upon their construction in the following theorem.

\begin{thm} (cf. \cite[Prop 2.1.1, Prop 2.2.2]{HNS}
\label{thm:HNS}
Let $\bG$ be a linear algebraic group defined over $\bF_p$, let $q = p^r$ be some power of $p$, and 
assume that $k$ contains $\bF_q$.  Let $H$ be a subgroup of $\bG(\bF_q)$ stable 
under the action of $F^r$, the $r$-th power of the Frobenius map.  
\begin{enumerate}
\item
Then  \ $ind^{\bG}_H(-): Mod(H) \quad \to \quad Mod(\bG)$ \
 takes values in mock injective $\bG$-modules.   
\item
If every irreducible $H$ module is the restriction of a $\bG$-module, 
then $ind^{\bG}_H(M)$ is a proper mock injective $\bG$-module if and only if the $M$ is not 
an injective $H$-module.
\end{enumerate}
\end{thm}

\begin{proof}
The proof that $ind^{\bG}_H(-)$ takes values in mock injective $\bG$-modules
is essentially given in \cite{HNS}; we give details of the formulation stated above.
We first observe that the restriction $(k[\bG]^R)_{H}$ of the $\bG$-module $k[\bG]^R$ to $H$ is 
an injective $H$-module since $\bG/H$ is affine (see \cite{CPS77}).  Moreover, if $I$ is an
injective $\bG$-module, then the embedding $Soc(I) \hookrightarrow Soc(I)^{tr} \otimes k[\bG]^R$
extends to an embedding $I  \hookrightarrow Soc(I)^{tr} \otimes k[\bG]^R$ which splits
as a map of $\bG$-modules, so that $I_{|H}$ must be injective as an $H$-module.
If $F^r: \bG \to \bG$ restricts to an isomorphism on $H$, then the restriction $L_{|H}$
to $H$ of the $\bG$-module $L \equiv (F^r)^*(k[\bG]^R)$ is an injective $H$-module;
this implies that $M \otimes L_{|H}$ is an injective $H$-module for any $H$-module $M$. 

The tensor identity (see \cite[I.3.6]{J}) tells us that for any $H$-module $M$ 
there is a natural isomorphism
\begin{equation}
\label{eqn:ind-iso}
ind_H^\bG(M) \otimes L \quad \simeq \quad ind_H^\bG(M \otimes L_{|H})
\end{equation}
of $\bG$-modules.  Since $L$ is trivial as a $\bG_{(r)}$-module, (\ref{eqn:ind-iso}) implies 
that $(ind_H^\bG(M))_{|\bG_{(r)}}$ is a direct summand of $(ind_H^\bG(M \otimes L_{|H}))_{|\bG_{(r)}}$. 
Since $L_{|H}$ is injective as an $H$-module, $M \otimes L_{|H}$ is injective as an $H$-module,
so that Proposition \ref{prop:ind-H}(2) tells us that $ind_H^\bG(M \otimes L_{|H})$ is injective as a $\bG$-module.
As seen above, this implies that $(ind_H^\bG(M \otimes L_{|H}))_{|\bG_{(r)}}$ is injective as a $\bG_{(r)}$-module.
Thus, the direct summand $(ind_H^\bG(M))_{|\bG_{(r)}}$ is also injective as a $\bG_{(r)}$-module.  Since this
applies to any $r>0$, we conclude that $ind_H^\bG(M)$ is a mock injective $\bG$-module.

Assertion (2) follows  from a simple argument using $Ext$-groups as detailed in \cite{HNS}.
\end{proof}

\vskip .1in

The following observation is of interest in view of the challenge of studying quotients of 
infinite dimensional $\bG$-modules as seen in Theorem \ref{thm:mock-quotients}.

\begin{cor}
\label{cor:short-exact}
As in Theorem \ref{thm:HNS}, consider a linear algebraic group $\bG$ defined over $\bF_p$,
assume $\bF_q \subset k$ with $q = p^r$, and consider a subgroup  $H \subset \bG(\bF_q)$ stable 
under the action of $F^r$.  If $0 \to N \to M \to Q \to 0$ is an exact sequence in $Mod(H)$, then
\begin{equation}
\label{eqn:mock-short}
0 \to \ ind^{\bG}_H(N) \quad \to \quad ind^{\bG}_H(M)  \quad \to \quad
ind^{\bG}_H(Q) \to 0.
\end{equation}
is a short exact sequence of mock injective $\bG$-modules.
\end{cor}

\vskip .1in

We make explicit the preceding discussion in the 
special case in which $\bG$ equals the additive group $\bG_a$.   
Notice that the third assertion shows how cofinite type can distinguish isomorphism classes
of mock injective modules, whereas the last assertion emphasizes that proper mock injective $\bG$-modules
are unlikely to have finite injective dimension.

\begin{ex}
\label{ex:Ga-mock}
Let $J_d$ denote the  proper mock injective $\bG_a$-module $ind^{\bG}_{ \bG_a(\bF_{p^d})}(k)$,
equipped with the natural $\bG_a$-equivariant embedding into $k[t]^L$
(which we denote by $\cO(\bG_a) \ = \ k[t]$ since $\bG_a$ is abelian).
\begin{enumerate}
\item
With respect to this embedding, $J_d$ can be identified as the subalgebra of $k[t]$
consisting of polynomials in $t^{p^d}-t$.  This is verified in greater generality in Proposition \ref{prop:ind}
below.
\item
$\cO(\bG_a)/J_d$ admits an embedding into $\cO(\bG_a)^{\oplus d}$.  Namely, the images of
$\{ t, t^p, \ldots, t^{p^{d-1}} \}$ in $\cO(\bG_a)/J_d$ form a basis for the socle of $\cO(\bG_a)/J_d$.  
\item
The cofinite type $\gamma(\bG_a,\phi)_{J_d}$ of the $\bG_a$-module $J_d$ equals $(1,q^{-1})$.   
This readily follows from (1) above.  On the other hand, the cofinite type of $\cO(\bG)/J_d$
equals $(q-1,q^{-1})$.
\item
In the special case $d = 1$, $J_1 \ = \ ind^{\bG}_{ \bG_a(\bF_p)}(k)$ has an injective resolution
$$J_1 \ \hookrightarrow \ I^0 \stackrel{d^1}{\to} I^1 \stackrel{d^2}{\to} I^2 \to \cdots 
\to I^{n-1} \stackrel{d^n}{\to} I^n \cdots $$
 with the property that each $I^n$ is isomorphic to $\cO(\bG_a)$ (i.e., the injective hull of $k$) and that
 $d^n = d^{n+2}$.  
 \end{enumerate}
\end{ex}

\vskip .1in

The identification of $ind_H^\bG(k)$ used in Example \ref{ex:Ga-mock} has the following
elaboration.  The key to the proof of this proposition is 
the Lang isomorphism $\bG/\bG(\bF_q) \ \stackrel{\sim}{\to} \ \bG$, an isomorphism of varieties 
over $k$ \cite{Lang}.

\begin{prop}
\label{prop:ind}
Let $q = p^r$ and assume that $\bF^q \subset k$.  Consider a 
linear algebraic group $\bG$ defined over $\bF_p$ provided with an 
embedding $\phi: \bG \hookrightarrow GL_N$ defined over $\bF_p$
and let $H$ denote the finite group $\bG(\bF_q)$.
Recall the Lang map
\begin{equation}
\label{eqn:Lang}
F^r/id: \bG \  \twoheadrightarrow \ \bG/H \ \simeq \ \bG, \quad g \mapsto F^r(g) \cdot g^{-1},
\end{equation}
a finite \`etale map of $k$-varieties (see \cite{Lang}). 
 
The map on coordinate algebras 
$$(F^r/id)^* = \mu \circ (\sigma_\bG \otimes (F^r)^*) \circ \Delta:  \cO(\bG) \ \to 
\cO(\bG) \otimes \cO(\bG) \to \cO(\bG), \quad f \ \mapsto  \sigma_\bG(f) \cdot f^r$$ 
has image the $\bG$-submodule $ind_H^\bG(k) = (k[\bG]^L)^H \subset k[\bG]^L \simeq k[\bG]^R$.
This implies that 
\begin{equation}
(k[\bG]^R)_{\leq \frac{d}{2N-1+q},\phi} \quad \subset \quad (ind_H^G(k))_{\leq d,\phi} 
\quad \subset \quad (k[\bG]^L)_{\leq d,\phi}.
\end{equation}
\end{prop}

\begin{proof}
As seen in the proof of Proposition \ref{prop:left-right}, $\sigma_{GL_N}(\cO(GL_N)_{\leq d}) \subset
\cO(GL_N)_{\leq (2N-1)\cdot d}$, so that
$\sigma_\bG(\cO(\bG)_{\leq d,\phi}) \ \subset \ \cO(\bG)_{\leq (2N-1)\cdot d,\phi}$.  Moreover, if 
$f \in \cO(GL_N)_{\leq d}$, then $(F^r)^*(f) \in \cO(GL_N)_{\leq qd}.$ 
Consequently,
$(F^r/id)^*(\cO(\bG)_{\leq d,\phi}) \ \subset \ \cO(\bG)_{\leq (2N-1+q)d,\phi}.$
\end{proof}

\vskip.1in

In the next example, we determine the cofinite type of the mock injective $\bU_N$-modules
$ind_{\bU_N(\bF_q)}^{\bU_N}(k) \subset k[\bU_N]^L$.  See Example \ref{ex:unipotent} for the
description of $\cO(\bU_N)_{\leq d,\phi}$.   For notational convenience, we omit explicit mention
of the natural inclusion $\phi: \bU_N \hookrightarrow GL_N$.

\begin{ex}
\label{ex:UN-mock}
Consider the antipode $\sigma_{\bU_N}: \cO(\bU_N) = k[x_{i,j}, i <j] \to k[y_{s,t}, s < t] = \cO(\bU_N)$.
Observe that $\sigma_{\bU_N}(\cO(\bU_N)_{\leq d}) \subset \cO(\bU_N)_{\leq (N-1)d}$,
that $\sigma_{\bU_N}(x_{1,N}) \notin \cO(\bU_N)_{\leq N-2}$, and that
$\sigma_{\bU_N}(x_{i,j}) \in \cO(\bU_N)_{\leq N-2}$ and $\sigma_{\bU_N}(x_{i,j}) \in 
\cO(\bU_N)_{\leq N-23} $for $(i,j) \not= (1,N)$.

To determine the cofinite type of $ind_{\bU_N(\bF_q)}^{\bU_N}(k)$ (which we denote by $M$), 
we consider the image of 
$$(F^r/id)^*: k[\bU_N]_{\leq d} \quad \to \quad M_{\leq (p^r+N-1)d} \ \subset \ k[\bU_N]_{\leq (p^r+N-1)d},$$
and proceed to show that $d \mapsto  dim(M_{\leq (p^r+N-1)d})$
has the same polynomial degree and the same leading coefficient as 
$d \mapsto dim((F^r/id)^*(k[\bU_N]_{\leq d})) = dim(k[\bU_N]_{\leq d})$
which was computed explicitly in Example \ref{ex:cofinite}.   
Namely, the above observations about the filtration degree of $\sigma_{\bU_N}(x_{i,j})$ imply
the short exact sequence 
$$0 \to (F^r/id)^*(k[\bU_N]_{\leq d}) \quad \to \quad M_{\leq (p^r+N-1)d} \quad \to \quad 
x_{1,N} \bu (F^r/id)^*(k[\bU_N]_{\leq d-1}) \to 0.$$

Consequently,
$$\gamma(\bU_N)_M \ = \ (N^\prime,((p^r+N-1)^{N^\prime}N!)^{-1}) \quad N^\prime = \frac{N^2-N}{2}.$$
\end{ex}

\vskip .1in

\begin{ex}
\label{ex:GL-mock}
One can apply the argument of Example \ref{ex:UN-mock} to obtain a similar computation for $GL_N$
replacing the structure of $\cO(\bU_N)$
as a polynomial algebra by the ``sum expansion" of $\cO(GL_N)$ given in (\ref{eqn:choose-sum}).
We then conclude for $M = ind_{GL_N(\bF_q)}^{GL_N}(k)$ that
$$\gamma(GL_N)_M \ = \ (N^2,((p^r+N-1)^{N^2}(N^2)!)^{-1}).$$
\end{ex}

\vskip .2in


\section{Stable Module Categories}
\label{sec:stable}

In Theorem \ref{thm:categories}, we provide a stable module-theoretic version of Theorem \ref{thm:HNS} 
by dividing out injective modules, thereby focusing upon proper mock injective modules.  We introduce
the tensor triangulated category $StMock(\bG)$, complementary to stable categories considered in \cite{F23}.

We remind the reader that the stable module category $StMod(G)$ of a finite group scheme $G$
over $k$ is the category whose objects are $G$-modules (i.e., $kG$-modules) and whose space of maps
$Hom_{StMod(G)}(M,N)$ for any pair $M, \ N$ of $G$-modules is the quotient of $Hom_G(M,N)$
by the subspace of maps $M \to N$ of $G$-modules which factor through an injective $G$-module.
Using the fact that $kG$ is self-injective, one provides $StMod(G)$ with the structure of a 
tensor triangulated category.  As shown in \cite[Thm 4.2]{F23}, there is a natural tensor triangulated equivalence 
\begin{equation}
\label{eqn:KbM}
StMod(G) \quad \stackrel{\sim}{\to} \quad K^b(Mod(G))/\cI nj(Mod(G)).
\end{equation}
natural with respect to finite group schemes $G$.  Here, $\cI nj^b(Mod(G))$ denotes the homotopy
category of bounded cochain complexes of $G$-modules which are quasi-isomorphic  to a 
bounded complex of injective $G$-modules.

\vskip .1in

With this in mind, we make the following definition, changing the use of the notation $StMod(\bG)$,
notation which was used in \cite{F23} for
the tensor triangulated category which we denote below by $\ol{StMod}(\bG)$.

\begin{defn}
\label{defn:StMod}
Let $\bG$ be a linear algebraic group.
Denote by $\cK^b(Mod(\bG))$ the tensor triangulated category given as the homotopy category of
 bounded cochain complexes of $\bG$-modules and denote by 
$\cI nj^b(\bG) \hookrightarrow \cK^b(Mod(\bG))$ the thick subcategory of bounded cochain 
complexes $C^\bu$ quasi-isomorphic to a bounded complex of injective $\bG$-modules.
We define $StMod(\bG)$ to be Verdier quotient
$$\cK^b(Mod(\bG)) \quad \twoheadrightarrow \quad \cK^b(Mod(\bG))/\cI nj^b(Mod(\bG)) 
\quad \equiv \quad StMod(\bG).$$

Denote by $\cM ock^b(\bG) \hookrightarrow \cK^b(Mod(\bG))$ the thick subcategory given as the 
homotopy category
of bounded cochain complexes $C^\bu$ whose restriction to $Mod(\bG_{(r)})$ is quasi-isomorphic
to a bounded complex of injective $\bG_{(r)}$-modules for each $r > 0$.
We define $StMock(\bG)$ to be Verdier quotient
$$\cM ock^b(\bG) \quad \twoheadrightarrow \quad \cM ock^b(\bG)/\cI nj^b(Mod(\bG)) \quad \equiv \quad StMock(\bG).$$

Finally, we define $\ol{StMod}(\bG)$ to be the Verdier quotient
\begin{equation}
\label{eqn:Verd-St}
StMock(\bG)  \quad \twoheadrightarrow \quad StMock(\bG)/\cM ock^b(\bG) \quad \equiv \quad \ol{StMod}(\bG).
\end{equation}
\end{defn}

\vskip .1in

The following proposition relates classes of $\bG$-modules 
to their associated classes in the appropriate Verdier quotient formulated using
bounded cochain complexes of $\bG$-modules.

\begin{prop}
Let $\bG$ be a linear algebraic group and $M$ a $\bG$-module determining
the chain complex $M[0] \in \cK^b(Mod(\bG))$.  Then the class of $M[0]$ in
$\ol{StMod}(\bG)$ is $0$ if and only if $M$ is a mock injective.  

Moreover, if
$M$ is a mock injective  $\bG$-module, the class in $StMock(\bG)$ of $M[0] \in 
\cM ock^b(\bG)$ is $0$ if and only if $M$ is an injective $\bG$-module. 
\end{prop}

\begin{proof}
If $M$ is a mock injective $\bG$-module, then essentially by definition
the class of $M[0]$ is 0.  Conversely, 
if the class of $M[0]$ in $\ol{StMod}(\bG)$ is $0$, then
$M[0] \in StMock(\bG)$.  In this case, the restriction to each $\bG_{(r)}$ of $M[0]$
is $0 \in StMock(\bG_{(r)})$.   By \cite[Cor 4.5]{F23}, this implies each $M_{|\bG_{(r)}}$ is an injective
$\bG_{(r)}$ module so that $M$ is a mock injective $\bG$-module.

The second assertion follows from the fact that the kernel of $\cM ock^b(\bG)
\to StMod(\bG)$ equals $\cK^b(Inj(\bG))$, since $\cK^b(Inj(\bG)) \ \subset \ \cM ock^b(\bG)$
is a thick subcategory.
\end{proof}

\vskip .1in

The following theorem summarizes the relationships between these various stable categories,
providing a form of ``categorization" of Theorem \ref{thm:HNS}

\begin{thm}
\label{thm:categories}
Let $\bG$ be a linear algebraic group defined over $\bF_p$, let $q = p^r$ be some power of $p$, and 
assume that $k$ contains $\bF_q$.  Let $H$ be a subgroup of $\bG(\bF_q)$ stable 
under the action of $F^r$, the $r$-th power of the Frobenius map.  Then there is a commutative
diagram of tensor triangulated categories and tensor triangulated functors
\begin{equation}
\label{eqn:sub-triang}
 \xymatrixcolsep{5pc}\xymatrix{
 \cI nj^b(Mod(H)) \ar[d]^{\hookrightarrow} \ar[r]^-{ind_H^\bG(-)} &  \cI nj^b(Mod(\bG)) \ar[d]^{\hookrightarrow} 
 \ar[r]^-= & \cI nj^b(Mod(\bG)) \ar[d]^{\hookrightarrow} \\
 D^b(Mod(H)) \ar[d]^\twoheadrightarrow  \ar[r]^-{ind_H^\bG(-)} & D^b(Mock(\bG))\ar[d]^\twoheadrightarrow 
   \ar[r]^-{\hookrightarrow} & D^b(Mod(\bG)) \ar[d]^\twoheadrightarrow  \\
StMod(H) \ar[r]^-{ind_H^\bG(-)}  & StMock(\bG)  \ar[dr]_0 \ar[r]^-{\hookrightarrow} & 
StMod(\bG) \ar[d]^\twoheadrightarrow  \\.
& & \ol{StMod}(\bG)
}
\end{equation}
where the derived categories are Verdier quotients of corresponding homotopy categories of cochain 
complexes ``localized" at subcategories of cochain complexes which are acylic.

The homotopy categories of injective modules in the upper row are the kernels of the Verdier quotients
of the functors from categories in the second row to categories in the third row of (\ref{eqn:sub-triang}).   
Moreover, the bottom right functor is the Verdier quotient with kernel the right horizontal map 
of the third row of (\ref{eqn:sub-triang}).
\end{thm}

\begin{proof}
The upper left horizontal functor is well defined because $ind_H^\bG(-)$ is exact and sends injective
$H$-modules to injective $\bG$-modules.  Exactness of $ind_H^\bG(-)$ justifies the middle left horizontal
functor.  The naturality of the construction of Verdier quotients justifies the lower left horizontal functor
as well as the commutativity of the left squares of (\ref{eqn:sub-triang}) follows from the definitions.  
The commutativity of the right squares of (\ref{eqn:sub-triang}) follows from Definition \ref{defn:StMod}.

The assertion about kernels of Verdier quotients follows from the thickness of the asserted kernels.
\end{proof}

\vskip .2in


\section{Criteria for injectivity of $\bG_a$-modules}
\label{sec:co-pi}

One goal for the study of mock injective $\bG$-modules is to refine the theory
of support varieties $M \mapsto \Pi(\bG)_M$ as considered in \cite{F23}
so that the support of $M$ is empty if and only if $M$ is an injective $\bG$-module.
In the very special case of $\bG = \bG_a$, Theorem \ref{thm:Ga-detect} provides such
a refinement which suggests a possible approach for more general $\bG$.

\begin{prop}
\label{prop:comm-coalg}
For any $r > 0$, there is a commutative square of commutative coalgebras
\begin{equation}
\label{eqn:comm-coalg}
 \xymatrixcolsep{5pc}\xymatrix{ 
\cO(\bG_a)_{\leq p^{r+1}-1} \ar[r] & \cO(\bG_{a(r+1)}) \ar[r] \ar[d]& (k\bZ/p^{\times r+1})^* \ar[d] \\
\cO(\bG_a)_{\leq p^r-1} \ar[r] \ar[u] & \cO(\bG_{a(r)}) \ar[r] & (k\bZ/p^{\times r})^*
}
\end{equation}
whose horizontal arrow are isomorphisms, whose
left vertical arrow is the natural inclusion, whose middle vertical arrow is induced by 
the embedding $\bG_{(r)} \hookrightarrow \bG_{(r+1)}$,
and whose right vertical arrow is the dual of the map group algebras induced by the
embedding $\bZ/p^{\times r} \to \bZ/p^{\times r+1}$ sending $(a_1,\ldots,a_r)$
to $(a_1,\ldots,a_r,0)$.  The left horizontal arrows are induced by the natural maps
$$\cO(\bG_a)_{\leq p^r-1} \to \cO(\bG_a)_{\leq p^{r+1}-1} \to \cO(\bG_a) \to  \cO(\bG_{a(r+1)}) 
\to \cO(\bG_{a(r)})$$
and the right horizontal maps are duals of familiar isomorphisms.
\begin{enumerate}
\item
For each $r>0$, there is a coalgebra splitting 
\begin{equation}
\label{eqn:split}
s_r: \cO(\bG_a)_{\leq p^{r+1}-1} \quad \twoheadrightarrow \quad \cO(\bG_a)_{\leq p^r-1} 
\end{equation}
of the natural inclusion $\cO(\bG_a)_{\leq p^r-1} \hookrightarrow \cO(\bG_a)_{\leq p^{r+1}-1}$
fitting in the commutative square
\begin{equation}
\label{eqn:comm-split}
 \xymatrixcolsep{5pc}\xymatrix{ 
\cO(\bG_a)_{\leq p^{r+1}-1} \ar[r] \ar[d]^{s_r} & \cO(\bG_{a(r+1)})  \ar[d] \\
\cO(\bG_a)_{\leq p^r-1} \ar[r]  & \cO(\bG_{a(r)}) 
}
\end{equation}
\item
For a $\bG_a$-module $M$ and each $r > 0$, let $M_r \subset M$ denote the subset consisting
of elements $m \in M$ such that the action of $\varinjlim_{s>r} k\bZ/p^{\times s}$ on $m$
factors through the projection $\varinjlim_{s>r} k\bZ/p^{\times s} \to k\bZ/p^{\times r}$.
Then the $\cO(\bG_a)_{\leq p^r-1}$-subcomodule $M_{\cO(\bG_a)_{\leq p^r-1}} \subset M$ is the
pull-back of  $M_r \subset M$ along the isomorphism 
$\cO(\bG_a)_{\leq p^r-1} \stackrel{\sim}{\to} (k\bZ/p^{\times r})^*$ of (\ref{eqn:comm-coalg}).  
 \end{enumerate}
\end{prop}

\begin{proof}
The left commutative square of (\ref{eqn:comm-coalg}) together with the fact that the
its horizontal map is an isomorphism determines the coalgebra splitting $s_r$ of
(\ref{eqn:split}) which fits in the commutative diagram (\ref{eqn:comm-split}).

The distribution algebra $k\bG_a$ is the polynomial algebra on countably infinitely
many $p$-nilpotent generators $k[u_0,u_1,\ldots,u_n \ldots]/(\{ u_i^p\})$.  The action of 
$u_r$ on an $\cO(\bG_a) = k[t]$-comodule $M$ with coproduct $\Delta_M: M \to M\otimes k[t]$
sends $m \in M$ to $(id_M\otimes \gamma_{p^r})(\Delta_M)(m))$, where $\gamma_{p^r}: k]t] \to k$
reads off the coefficient of $t^{p^r}$ of each $f (t) \in k[t]$.  Thus, $M_{\leq p^r-1} \subset M$
consists of those elements $m \in M$ such that $u_n$ acts trivially for $n \geq r$.   Indexing
the factors of $\bZ/p^{\times r}$ as $\prod_{s= 0 }^r \bZ/p$, we conclude that the pull-back
of $M_r \subset M$ along the isomorphism of coalgebras 
$\cO(\bG)_{\leq p^r-1} \stackrel{\sim}{\to} (k\bZ/p^{\times r})^*$ equals  $M_{\leq p^r-1}$.
\end{proof}

\vskip .1in

Proposition \ref{prop:comm-coalg}(2) together with the commutativity of (\ref{eqn:comm-coalg})
easily implies the following corollary
which gives a concrete description of $\bG_a$-modules in terms of a sequence of 
modules for elementary abelian $p$-groups.

\begin{cor}
\label{cor:Mr}
The isomorphisms of  (\ref{eqn:comm-coalg}) determine a 1-1 correspondence between 
the data of a filtered $\bG_a$-module $M = \varinjlim_r M_{\leq p^r-1}$ 
and the data of a sequence 
$M_1 \subset \cdots \subset M_r \subset M_{r+1} \subset \cdots$ of $k$ vector spaces 
with each $M_r$ equipped with an action of $k\bZ/p^{\times r}$ such that 
$M_r \subset M_{r+1}$ consists of those $m \in M_{r+1}$ on which the action of $k\bZ/p^{\times r+1}$
factors through the projection onto the first $r$ factors, $k\bZ/p^{\times r+1} \twoheadrightarrow
k\bZ/p^{\times r}$ inducing the action of $k\bZ/p^{\times r}$ on $M_r$.
\end{cor}

\vskip .1in

Combining Proposition \ref{prop:comm-coalg} and Corollary \ref{cor:Mr}, we obtain the following
proposition.

\begin{prop}
\label{prop:Pi-inj}
Retain the notation of Proposition \ref{prop:comm-coalg}.  Then the the map \\
$\Pi(\bZ/p^{\times r}) \to \Pi(\bZ/p^{\times r+1})$ induced by the embedding of finite groups
$\bZ/p^{\times r} \hookrightarrow \bZ/p^{\times r+1}$ which sends $(a_1,\ldots,a_r)$ to $(a_1,\ldots,a_r,0)$
restricts to 
\begin{equation}
\label{eqn:restrict-Pi}
\Pi(\bZ/p^{\times r})_{M_r} \  = \ \Pi(\bZ/p^{\times r+1})_{M_{r+1}} \cap \Pi(\bZ/p^{\times r})
\ \hookrightarrow \ \Pi(\bZ/p^{\times r+1})_{M_{r+1}}.
\end{equation}

Moreover, there is a commutative square 
\begin{equation}
\label{eqn:comm-split2}
 \xymatrixcolsep{5pc}\xymatrix{ 
 \Pi(\bG_{(r)})_{M_{< p^{r}-1}} \ar[r]^\simeq \ar[d] & \Pi(\bZ/p^{\times r})_{M_{r}} \ar[d]\\
 \Pi(\bG_{(r+1)})_{M_{< p^{r+1}-1}} \ar[r]^\simeq  & \Pi(\bZ/p^{\times r+1})_{M_{r+1}} 
}
\end{equation}
whose horizontal bijections are given by Proposition \ref{prop:comm-coalg}(2) and 
whose left vertical arrow is the composition of the map $\Pi(\bG_{(r)})_{M_{< p^{r}-1}} 
\ \to \ \Pi(\bG_{(r+1)})_{M_{< p^{r}-1}} $ associated to the splitting $s_r$ of (\ref{eqn:split}) 
followed by the map
$\Pi(\bG_{(r+1)})_{M_{< p^{r}-1}}  \ \to \  \Pi(\bG_{(r+1)})_{M_{< p^{r+1}-1}} $
enabled by (\ref{eqn:restrict-Pi}).
\end{prop}

\begin{proof}
Let $\alpha: K[t]/t^p \to K\bZ/p^{\times r}$ be a flat map representing a $\pi$-point in $\Pi(\bZ/p^{\times r}).$
Than $\alpha$ represents a $\pi$ point of $\Pi(\bZ/p^{\times r})_{M_r}$ if and only if 
$\alpha^*((M_r)_K)$ is not a free $K[t]/t^p$-module.  The composition of $\alpha$ with 
the embedding $i_r: K\bZ/p^{\times r} \hookrightarrow K\bZ/p^{\times r+1}$ has the property that
$(i_r \circ \alpha)^*(M_{r+1})_K) = \alpha^*((M_r)_K)$.  Thus, $i_r \circ \alpha$ represents 
a $\pi$-point in $\Pi(\bZ/p^{\times r})_{M_{r+1}}$ if and only if  $\alpha$ represents a $\pi$ point of 
$\Pi(\bZ/p^{\times r})_{M_r}$

The verification of the commutativity of the square (\ref{eqn:comm-split2}) follows from the definitions
of the maps involved.
\end{proof}

\vskip .1in

The following criterion for injectivity for a $\bG_a$-module might possibly permit a generalization
linear algebraic groups other than $\bG = \bG_a$.

\begin{thm}
\label{thm:Ga-detect}
Retain the notation of Proposition \ref{prop:comm-coalg}. Then the following are equivalent.
\begin{enumerate}
\item
$M$ is an injective $\bG_a$-module $M$.
\item
$\varinjlim_r \Pi(\bG_{(r)})_{M_{\cO(\bG_a)_{\leq p^r-1}}}$ is empty.
\item
$\varinjlim_r \Pi(\bZ/p^{\times r})_{M_r}$ is empty.
\end{enumerate}
\end{thm}

\begin{proof}
By Proposition \ref{prop:X-injectivity} and Proposition \ref{prop:advantages}, $M$ is injective 
if and only if $M_{\cO(\bG_a)_{\leq p^r-1}}$ is
an injective object of $CoMod(\cO(\bG_a)_{\leq p^r-1})$ for all $r > 0$.  
The injectivity assertion of Proposition \ref{prop:Pi-inj} together with Proposition
(\ref{prop:comm-coalg}) implies the equivalence of conditions (1) and (3).
The equivalence of (2) and (3) now follows from (\ref{eqn:comm-split2}).
\end{proof}

\vskip .2in


\section{Questions}
\label{sec:questions}

We mention a few of the many questions which arise when considering 
mock injective $\bG$-modules.
\vskip .1in

\begin{question}
\label{ques:fundamental}
Is there a useful ``geometric invariant" for mock $\bG$-modules which complements support varieties
by detecting injectivity of $\bG$-modules?  (See Theorem \ref{thm:Ga-detect} above.)
\end{question}

\vskip .1in
\begin{question}
\label{ques:Hoch}
For $\bU$ a unipotent linear algebraic group, can we utilize ascending, converging sequences
of subcoalgebras of $\cO(\bU)$ to enable computations
of the Hochschild cohomology of $\bU$?  (See \cite{F19}.)
\end{question}

\vskip .1in

\begin{question}
\label{ques:new-mocks}
If $M$ is a cofinite mock injective $\bG$-modules, does $M$ admit a filtration by 
mock injective $\bG$-modules which are obtained by inducing $H$-modules to $\bG$
for finite subgroups $H \hookrightarrow \bG$?
\end{question}

\vskip .1in
\begin{question}
\label{ques:quotient-mock}
Under what hypotheses does a cofinite mock injective $\bG$-module $J$ admit an embedding
into an injective $\bG$-module $I$ with quotient $I/J$ which is also cofinite?
\end{question}

\vskip .1in
\begin{question}
\label{ques:test}
Let $\bG$ be semi-simple, simply connected algebraic group defined over $\bF_p$.  If $J$ is 
a mock injective $\bG$-module whose restriction to each $\bG(\bF_q)$ is injective, then is $J$
an injective $\bG$-module?
\end{question}

\vskip .2in



\begin{thebibliography}{20} 


\bibitem{BIKP} D. Benson, S. Iyengar, H. Krause, J. Pevtsova, {\em Rank varieties and $\pi$-points
for elementary supergroup schemes}, Trans. Amer. Math. Soc. Ser. B {\bf 8} (2021), 971 - 998.


\bibitem{C} J. Carlson, {\em The varieties and the cohomology ring of a module}, J. Algebra
{\bf 85} (1983), 104 -143.

\bibitem{CF} J. Carlson, E. Friedlander {\em Exact category of modules of constant Jordan type}, 
Manin Festscrift, Progr. Math {\bf 269},  (2009), 267 - 290.


\bibitem{CPS77} E. Cline, B. Parshall, L. Scott,  {\em Induced Modules and Affine Quotients}, Math. Ann. 
{\bf 30} (1977), no. 1, 1--14.




\bibitem{Donkin} S. Donkin, {On projective modules for algebraic groups}, J. London Math. Soc {\bf 54} 
(1996), 75 - 88.

\bibitem{Doty} S. Doty, {\em Polynomial representations, algebraic monoids, and Schur algebras of classical
type}, J.Pure and Appl. Algebra {\bf 123} (1998), 165 - 199.

\bibitem{F11} E. Friedlander, {\em Restrictions to $\bG(\bF_p)$ and $\bG_{(r)}$ of rational $\bG$-modules},
Compositio Math. {\bf 147} (2011), 1955 - 1978..

\bibitem{F15} E. Friedlander, {\em Support varieties for rational representations,} Compositio Math {\bf 151} 
(2015), 765-792.
%


\bibitem{F18} E. Friedlander, {\em Filtrations, 1-parameter subgroups, and rational injectivity}, Advances in Math
{\bf 323} (2018), 84 - 113.

\bibitem{F19} E. Friedlander, {\em Cohomology of unipotent group schemes}, Algebr. Represent Theory {\bf 22}
(2019), 1427 - 1455.

\bibitem{F23} E. Friedlander, {\em Support varieties and stable categories for algebraic groups},
Compositio Math. {\bf 159} (2023), 746 - 779.


\bibitem{F-N2} E. Friedlander and C. Negron, {\em Support theory for Drinfeld doubles for some 
infinitesimal group schemes}, Algebra Number Theory {\bf 17} (2023), 217-260.


\bibitem{FP2} E. Friedlander, J. Pevtsova, {\em $\Pi$-supports for modules 
for finite group schemes},  Duke. Math. J. {\bf 139} (2007), 317--368.


\bibitem{FS} E. Friedlander, A. Suslin, {Cohomology of finite group schemes over a field}, Invent. Math. 
{\bf 127}(1997), 209-270.

\bibitem{Green} J.A. Green, {\em Polynomial representations of $GL_n$; 2nd ed}, Lecture Notes in Mathematics
{\bf vol 830}, Springer Berline, 2007.




\bibitem{HNS}  W. Hardesty, D. Nakano, P. Sobaje, {\em On the existence of mock injective modules for
algebraic groups}, Bull. Lond. Math. Soc. {\bf 49} (2017), 806-817.

\bibitem{J2}  J.C. Jantzen, {\em Darstellungen halbeinfacher Gruppen und ihrer Frobenius-Kerne}
J. Reine Angew. Math. 317 (1980), 157–199.

\bibitem{J} J.C. Jantzen, {\em Representations of Algebraic groups -- 2nd ed}, Mathematical Surveys and 
Monographs {\bf 107}, Amer. Math Soc., 2003.

\bibitem{Lang}  S. Lang, {\em Algebraic groups over finite fields}, Amer. J. Jath. {\bf 78} (1956), 555--563.

\bibitem{L-N} Z. Lin, D. Nakano, {\em Complexity for modules for finite Chevalley groups and classical Lie
algebras}, Invent. Math. {\bf 138} (1999), 85 - 101.
%
%

\bibitem{smontgom} S. Montgomery, {\em Hopf algebras and their actions on rings}, CBMS Regional Conf. Series in Math {\bf 82}, Amer. Math. Soc., Providence, R.I. 1983.

\bibitem{NP} C. Negron, J. Pevtsova, {\em Hypersurface support for non-commutative complete interesections}, 
Nagoya Math. J. {bf 247} (2022), 731-750.

\bibitem{Pevt} J. Pevtsova, {\em Infinite dimensional modules for Frobenius kernels}, J.Pure and Appl. 
Algebra {\bf 173} (2002), 59 - 86.

\bibitem{Quil}  D. Quillen, {\em Higher K-theory, I}, Algebraic K-theory I, Lecture Notes in Mathematics
{\bf vol 341}, Springer Berlin (1973), 85 - 147.




\bibitem{Sobj13} P. Sobaje,  {\em On exponentiation and infinitesimal one-parameter subgroups of reductive groups},
J. Algebra {\bf 385} (2013), 14-26.



\bibitem{SFB1} A. Suslin, E. Friedlander, C. Bendel,
{\em Infinitesimal 1-parameter subgroups and cohomology},
J. Amer. Math. Soc. {\bf 10} (1997), 693-728.

\bibitem{SFB2} A. Suslin, E. Friedlander, C. Bendel,
{\em Support varieties for infinitesimal group schemes},
J. Amer. Math. Soc. {\bf 10} (1997), 729-759.

\bibitem{Sw} M. Sweedler, {\em Hopf Algebras}, Benjamin, New York, 1969.





\end{thebibliography}
\end{document}